\documentclass{article}

\usepackage{PRIMEarxiv}

\usepackage[utf8]{inputenc} 
\usepackage[T1]{fontenc}    
\usepackage{hyperref}       
\usepackage{url}            
\usepackage{booktabs}       
\usepackage{amsfonts}       
\usepackage{nicefrac}       
\usepackage{microtype}      
\usepackage{lipsum}
\usepackage{fancyhdr}       
\usepackage{graphicx}       
\graphicspath{{media/}}     

\usepackage{colortbl}  
\usepackage{makecell}  
\usepackage{multirow}
\usepackage[normalem]{ulem} 
\usepackage{caption}
\usepackage{subcaption}
\usepackage{pifont}
\usepackage{mathtools}
\usepackage{ytableau}
\usepackage{float}  
\usepackage{bm}

\usepackage{tikz}
\usetikzlibrary{calc}      
\usetikzlibrary{patterns}  
\usetikzlibrary{arrows}
\usetikzlibrary{matrix}
\usetikzlibrary{decorations.pathreplacing,calligraphy}

\usepackage{amsmath,amsthm,amssymb}
\allowdisplaybreaks
\theoremstyle{definition}

\newtheorem{theorem}{Theorem}
\newtheorem{lemma}{Lemma}

\DeclareMathOperator*{\argmin}{arg\,min}

\title{Fr\'echet Covariance and MANOVA Tests for Random Objects in Multiple Metric Spaces
}

\author{
  Alex Fout \\
  Department of Statistics \\
  Colorado State University\\
  Fort Collins, Colorado\\
  \texttt{fout@colostate.edu} \\
   \And
  Bailey K. Fosdick \\
  Department of Biostatistics and Informatics\\
  University of Colorado, Anschutz Medical Campus \\
  Aurora, Colorado\\
  \texttt{bailey.fosdick@cuanschutz.edu} \\
}

\begin{document}
\maketitle

\begin{abstract}
In this manuscript we consider random objects being measured in multiple metric spaces, which may arise when those objects may be measured in multiple distinct ways.
In this new multivariate setting, we define a Fr\'echet covariance and Fr\'echet correlation in two metric spaces, and a Fr\'echet covariance matrix and Fr\'echet correlation matrix in an arbitrary number of metric spaces.
We prove consistency for the sample Fr\'echet covariance, and propose several tests to compare the means and covariance matrices between two or more groups.
Lastly, we investigate the power and Type I error of each test under a variety of scenarios.
\end{abstract}


\section{Introduction}

Much networks research has traditionally addressed questions pertaining to a single observed network~\cite{newman_structure_2003, kolaczyk_statistical_2014}, focusing on either global characteristics (e.g. degree distribution, number of triadic relationships) or local characteristics (e.g. node attributes, individual dyadic relationships).
Statistics in this context has often been applied at the local scale by treating each node or dyadic relationship as an observation and addressing the dependencies between observations implied by the network structure~\cite{marrs_regression_2019, marrs_standard_2017}.
At the global scale, however, there is only one network, which is a single observation, so there is little opportunity for statistical modeling without placing assumptions on the generating process~\cite{erdos_evolution_1960, albert_statistical_2002, watts_collective_1998}.

However, the advent of network data comprised of multiple networks led researchers to readdress basic statistical concepts in this new setting, such as the sample mean network, population mean network, and measurement errors~\cite{kolaczyk_averages_2020, ginestet_hypothesis_2017}.
With observed networks themselves now being viewed as realizations of random networks having some probability measure, tools from shape analysis, non-Euclidean data sets, and object oriented data analysis were leveraged~\cite{wang_object_2007, bhattacharya_nonparametric_2012, ginestet_strong_2012, feragen_8_2020}.
These tools, such as the notion of a Fr\'echet mean\cite{frechet_les_1948}, have been used to redefine the concepts of consistency and convergence for random objects.
This has led to the development of methods analogous to traditional two sample testing, ANOVA, and regression in service of practitioners dealing with random object data sets~\cite{petersen_frechet_2016}.

The purpose of this manuscript is to continue in the vein of building basic analysis tools for these more complicated random objects in bounded metric spaces.
The appeal of metric spaces is that they encompass several examples of non-Euclidean random objects, including networks, probability distributions, and other multivariate objects~\cite{dubey_frechet_2019}.
On the other hand, metric spaces do not admit an arithmetic mean, voiding many routine statistical models and testing procedures.
In contrast to the Euclidean setting, where the observations in a sample each have an absolute position, a sample in metric space can produce only a set of pairwise distances between observations, with no other information about the underlying geometry, if there is one.
Without an arithmetic mean, statisticians have turned to the Fr\'echet mean for a basis from which to build anew classical statistical methods~\cite{dubey_frechet_2019, petersen_frechet_2019}.
The next section reviews one such development, the analysis of variance, in preparation for our introduction of a multivariate extension.

In this manuscript we propose a definition for Fr\'echet covariance and establish consistency of the sample estimator.
We also discuss various definitions of the Fr\'echet correlation, with particular attention paid to the interpretation of each.
We then propose several statistical tests for differences in group mean and covariance matrices among two or more groups of random objects in bounded metric spaces.
Some test statistics draw inspiration from existing MANOVA tests, others rely on the Riemannian geometry of symmetric positive definite matrices.
In two scenarios, we perform simulation studies to assess the Type I error of each statistic, as well as the power under various departures from the null hypothesis.

\section{Analysis of Variance in Bounded Metric Spaces}

Our work builds upon that of Dubey \& M\"uller~\cite{dubey_frechet_2019}, which studies Fr\'echet variance and an ANOVA procedure for differences in means and variances between two or more groups.
In this section we briefly review their work, though we have adjusted the notation to be consistent with the rest of this manuscript.
We begin with the following definitions.
Let $\Omega_s$ be a non-empty metric space with metric $d_s$ and generic point $\omega \in \Omega_s$.
We have added the $s$ subscripts in preparation for adding more metric spaces later.
Suppose random object $X_s \in \Omega_s$ has probability measure $P$.
Then the population Fr\'echet mean is given by
\begin{equation}
    \mu_s = \underset{\omega \in \Omega_s} \argmin E[d_s^2(\omega, X_s)],
\end{equation}
and for an independent and identically distributed sample $\{{X_{si}}\}$ with $i\in \{1, \ldots, n\}$, each $X_{si}$ having probability measure $P$, the corresponding sample mean is
\begin{equation}
    \hat{\mu}_s = \underset{\omega \in \Omega_s} \argmin \frac{1}{n} \sum_{i=1}^{n} d_s^2(\omega, X_{si}).
\end{equation}
If the metric space happens to be Euclidean space equipped with the Euclidean $L^2$ metric, then the Fr\'echet mean is the same as the arithmetic mean.
The Fr\'echet variance is given by
\begin{equation}
    \sigma_s^2 = E[d_s^2(\mu_s, X_{s})] \label{eq:fvar}
\end{equation}
and the corresponding sample variance is
\begin{equation}
    \widehat{\sigma_s^2} = \frac{1}{n} \sum_{i = 1}^n d_s^2(\hat{\mu}_s, X_{si}). \label{eq:fvar_samp}
\end{equation}

Petersen \& M\"uller established the consistency of the sample Fr\'echet mean~\cite{petersen_frechet_2019}, and Dubey \& M\"uller established the consistency and asymptotic normality of the sample Fr\'echet variance~\cite{dubey_frechet_2019}.
From these results, Dubey \& M\"uller developed an ANOVA procedure to test for equality of variance and means simultaneously among several groups for a random variable defined on a metric space.

Dubey \& M\"uller's test combines two statistics, $F_s$ and $U_s$, where $F_s$ is analogous to the $F_s$ statistic from classical ANOVA tests, and $U_s$ (which is a ``U-statistic'' in the sense of~\cite{hoeffding_class_1948}) tests for equality of variances.
While typically equality of group variances is an assumption to be checked before performing the ANOVA procedure (for example, see~\cite{levene_robust_1960}), the metric space ANOVA tests against a null of the means and variances being equal among the groups. 
One advantage of this approach is that it produces an asymptotic test thanks to the asymptotic distribution of $U_s$ and the fact that $F_s$ is $o_p(n^{-1/2})$.
Their proposed test also does not require normality assumptions of the groups, which seems appropriate given the diversity of random objects which are found it the setting of bounded metric spaces.
We outline the major components of Dubey \& M\"uller's test below insofar as they are relevant to our approach later.

Given $j$ groups denoted by $j \in \{1, \ldots, J\}$, let observation $X_{jsi}$ be the $i$th observation from group $j$, where $i \in \{1, \ldots, n_j\}$, and assume $\{X_{jsi}\}$ are independent across all $i$ and $j$, and identically distributed within each group $j$.
Define the group specific means and variances as 
\begin{align}
    \mu_{js} &= \underset{\omega \in \Omega_s} \argmin E[d_s^2(\omega, X_{jsi})], \quad j\in \{1, \ldots, J\} \\
    \sigma_{js}^2 &= E[d_s^2(\mu, X_{jsi})],
\end{align}
and corresponding sample estimators as
\begin{align}
    \hat{\mu}_{js} &= \underset{\omega \in \Omega_s} \argmin \frac{1}{n_j} \sum_{i=1}^{n_j} d_s^2(\omega, X_{jsi}), & j\in \{1, \ldots, J\},\\
    \widehat{\sigma_{js}^2} &= \frac{1}{n_j} \sum_{i = 1}^{n_j} d_s^2(\hat{\mu}_{js}, X_{jsi}) & j\in \{1, \ldots, J\}.
\end{align}
The hypothesis being tested is
\begin{align}
 \text{H}_0:\quad\mu_{1s} = \ldots = \mu_{Js},\quad \sigma_{1s}^2 = \ldots = \sigma_{Js}^2.
\end{align}
Define the pooled sample mean and variance as
\begin{align}
    \hat{\mu}_{ps} &= \underset{\omega \in \Omega_s} \argmin \frac{1}{n_j} \sum_{j=1}^J \sum_{i=1}^{n_j} d_s^2(\omega, X_{jsi}), \\
    \widehat{\sigma_{ps}^2} &= \sum_{j=1}^J \frac{1}{n_j} \sum_{i = 1}^{n_j} d_s^2(\hat{\mu}_{js}, X_{jsi}).
\end{align}
Since the sample sizes are not necessarily equal, let $\gamma_j = n_j / n$ represent the proportion of observations belonging to group $j$.
Then define the statistics 
\begin{align}
    F_s &= \widehat{\sigma_{ps}^2} - \sum_{j=1}^J \gamma_j \widehat{\sigma_{js}^2} \\
    U_s &= \sum_{j < j'} \frac{\gamma_j \gamma_{j'}}{\widehat{\text{Var}}(\widehat{\sigma_{js}^2}) \widehat{\text{Var}}(\widehat{\sigma_{j's}^2})} \left( \widehat{\sigma_{j}^2} - \widehat{\sigma_{j'}^2} \right)^2,
\end{align}
where 
\begin{align}
    \widehat{\text{Var}}(\widehat{\sigma_{js}^2}) = \frac{1}{n_j} \sum_{i = 1}^{n_j} d_s^4(\hat{\mu}_{js}, X_{jsi}) - \left\{ \frac{1}{n_j} \sum_{i = 1}^{n_j} d_s^2(\hat{\mu}_{js}, X_{jsi}) \right\}^2, \quad j\in \{1, \ldots, J\}
\end{align}
is an estimate of the variance of $\widehat{\sigma_{js}^2}$.
The $F_s$ statistic can be seen as estimating the between group variances by subtracting the average within group variance from the the total variance, and is analogous to the numerator of the Euclidean ANOVA statistic.
The $U_s$ statistic is similar to the test statistic underlying Levene's test, and when appropriately scaled, is asymptotically $\chi^2_{(J-1)}$ as $n\rightarrow \infty$, assuming each $\gamma_j$ converges as well.
To perform the metric space ANOVA, Dubey \& M\"uller propose a test statistic $T_s$ defined as
\begin{align}
    T_s = \frac{nU_s}{\sum_{j=1}^J \frac{\gamma_j}{\widehat{\text{Var}}(\widehat{\sigma_{js}^2})}} + \frac{nF_s^2}{\sum_{j=1}^J \gamma_j^2 \widehat{\text{Var}}(\widehat{\sigma_{js}^2})}\label{eq:frechet_anova}.
\end{align}
Under some existence and uniqueness assumptions for the population means and variances, $T_s$ converges in distribution to $\chi^2_{(J-1)}$ under the null hypothesis as $n \rightarrow \infty$.

Simulation studies show that Dubey \& M\"uller's metric space ANOVA performs well when considering scale-free networks, probability distributions, and even multivariate Euclidean data (each object being an observation in $\mathbb{R}^5$).
Our purpose in the rest of this manuscript is to extend this approach to the setting of two or more metric spaces, to test for differences in two or more groups measured in multiple metric spaces, analogous to a metric space MANOVA procedure.
We begin this with the introduction of Fr\'echet covariance in the next section.

\section{Fr\'echet Covariance}

Consider now two distinct metric spaces, $\Omega_s$ with metric $d_s$ and $\Omega_{s'}$ with metric $d_{s'}$, and random object vector $(X_s, X_{s'}) \in \Omega_s \times \Omega_{s'}$ with joint distribution $P_{ss'}$.
Strictly speaking, $(X_s, X_{s'})$ is an ordered pair not necessarily belonging to a vector space, but we use the term ``vector'' for convenience.
For compactness we will omit the subscripts of the metrics $d_s$ and $d_{s'}$ and use simply $d$ when it is obvious from context which one we mean.
In particular, $d_s$ measures distances in $\Omega_s$ and $d_{s'}$ measures distances in $\Omega_{s'}$.
We define the population mean vector as $(\mu_s, \mu_{s'})$, with elements given by:
\begin{align}
    \mu_* &= \underset{\omega \in \Omega_*} \argmin E[d^2(\omega, X_*)], \quad * \in \{s,s'\}.
\end{align}
Given a set of independent and identically distributed observations $\{(X_{si}, X_{s'i})\}$ with $i \in \{1, \ldots, n\}$, the sample mean vector $(\hat{\mu}_s, \hat{\mu}_{s'})$ has elements 
\begin{align}
    \hat{\mu}_* &= \underset{\omega \in \Omega_{*}} \argmin \frac{1}{n} \sum_{i=1}^{n} d^2(\omega, X_{*i}), \quad * \in \{s,s'\}.
\end{align}
The population and sample variance vectors, $(\sigma_s^2, \sigma_{s'}^2)$ and $(\widehat{\sigma_s^2}, \widehat{\sigma_{s'}^2})$, are defined analogously to \eqref{eq:fvar}-\eqref{eq:fvar_samp} with elements
\begin{align}
    \sigma_* &= \underset{\omega \in \Omega_{*}} \argmin E[d^2(\omega, X_*)] & *\in \{s,s'\}, \\
    \widehat{\sigma_*^2} &= \frac{1}{n} \sum_{i = 1}^n d^2(\hat{\mu_*}, X_{*i}) & *\in \{s,s'\}.
\end{align}
Now that there are multiple metric spaces, we may consider the dependency which may exist between $X_s$ and $X_{s'}$, the two elements of the random object vector.
Before proposing how to do so, however, we first address exactly what dependence means in a two-metric-space context.

Because general metric spaces lack a coordinate system which endows direction, it does not make sense to talk about the tendency for $X_s$ to \emph{increase} or \emph{decrease} as $X_{s'}$ \emph{increases}.
However, $X_s$ and $X_{s'}$ can be characterized by their distance from their respective means, and one sensible notion of dependence would be whether $X_s$ tends to move away from or toward its mean as $X_{s'}$ moves away from its mean.
This suggests the following definitions for the population and sample Fr\'echet covariance:
\begin{align}
    \sigma_{ss'} &= E[d(X_s, \mu_s) d(X_{s'}, \mu_{s'})]
    \label{eq:cov_def1}\\
    \widehat{\sigma}_{ss'} &= \frac{1}{n} \sum_{i=1}^n d(X_{si}, \hat{\mu}_s) d(X_{s'i}, \hat{\mu}_{s'}),
    \label{eq:cov_def1_samp}
\end{align}
which we expect to be large if $X_s$ and $X_{s'}$ tend to deviate from their mean together, and small if not.
This definition, however, always produces a positive covariance, since distances are always positive, so it does not clearly indicate when $X_s$ and $X_{s'}$ both tend to move away from their respective means at the same time, or if one tends to move away from its mean as the other approaches its mean.
Indeed, unlike the Fr\'echet variance, which reduces to the classical definition when applied to Euclidean objects with the Euclidean metric, the definition in \eqref{eq:cov_def1} reduces to $E[|X_s-\mu_s||X_{s'}-\mu_{s'}|]$, which is not the Euclidean covariance.

Using the definition for covariance in \eqref{eq:cov_def1} allows us to express the covariance matrix of $(X_s,X_{s'})$ as $\Sigma = E[d_\mu(X_s,X_{s'}) d_\mu(X_s,X_{s'})^T]$, where the distance vector $d_\mu(X_s,X_{s'})^T\allowbreak =\allowbreak (d(X_s, \mu_s),\allowbreak d(X_{s'}, \mu_{s'}))$.
Both the population and sample versions are positive semi-definite, positive definite if $X_s$ and $X_{s'}$ are non-degenerate, as shown later in Lemma \ref{lem:cov_mat_psd}.

An alternative definition of covariance is motivated by the desire that the sign of the covariance be meaningful, which we call the \emph{centered Fr\'echet covariance}:
\begin{align}
    \sigma_{ss'}^{'} &= E\left\{ \left[ d(X_s, \mu_s) - E(d(X_s, \mu_s)) \right] \left[ d(X_{s'}, \mu_{s'}) - E(d(X_{s'}, \mu_{s'})) \right] \right\}
    \label{eq:cov_def2}\\
    \widehat{\sigma}_{ss'}^{'} &= \frac{1}{n}\sum_{i=1}^n \left[d(X_{si}, \hat{\mu}_s) - \overline{d}_s\right] \left[d(X_{s'i}, \hat{\mu}_{s'}) - \overline{d}_{s'}\right]
    \label{eq:cov_def2_samp}
\end{align}
where $\overline{d}_s = \frac{1}{n}\sum_{i=1}^n d(X_{si}, \hat{\mu}_s)$ and $\overline{d}_{s'} = \frac{1}{n}\sum_{i=1}^n d(X_{s'i}, \hat{\mu}_{s'})$.
Equation \eqref{eq:cov_def2} is simply the Euclidean covariance between $d(X_s, \mu_s)$ and $d(X_{s'}, \mu_{s'})$, and \eqref{eq:cov_def2_samp} is the sample analog.
If both the $X_s$ and $X_{s'}$ distances increase / decrease together, this will be positive.
Unlike \eqref{eq:cov_def1} and \eqref{eq:cov_def1_samp}, equations \eqref{eq:cov_def2} and \eqref{eq:cov_def2_samp} can be negative, when the $X_s$ distances and $X_{s'}$ distances have an inverse relationship.
This is arguably the nicer interpretation, though it does not align as well with Dubey \& M\"uller's definition of Fr\'echet variance in \eqref{eq:fvar}.

For both covariance definitions, we can define correlation by normalizing by the appropriate standard deviations.
For the definitions in \eqref{eq:cov_def1} and \eqref{eq:cov_def1_samp}, we divide by the square root of Dubey \& M\"uller's variance to produce the \emph{non-centered Fr\'echet correlation}:
\begin{align}
    \rho_{ss'} &= \frac{E[d(X_s, \mu_s) d(X_{s'}, \mu_{s'})]}{\sqrt{E[d^2(X_s, \mu_s)] E[d^2(X_{s'}, \mu_{s'})]}}, \label{eq:cor_def1} \\
    \hat{\rho}_{ss'} &= \frac{\frac{1}{n}\sum_{i=1}^n d(X_{si}, \hat{\mu}_s) d(X_{s'i}, \hat{\mu}_{s'})}{\sqrt{\frac{1}{n}\sum_{i=1}^n d^2(X_{si}, \hat{\mu}_s) \frac{1}{n}\sum_{i=1}^n d^2(X_{s'i}, \hat{\mu}_{s'})]}}.
    \label{eq:cor_def1_samp}
\end{align}
For the definitions in \eqref{eq:cov_def2} and \eqref{eq:cov_def2_samp}, we simply use the Euclidean version of correlation on the distances to produce the \emph{centered Fr\'echet correlation}:
\begin{align}
    \rho_{ss'}' &= \frac{E\left[\left(d(X_s, \mu_s)-Ed(X, \mu_s)\right) \left(d(X_{s'}, \mu_{s'}) - Ed(X_{s'}, \mu_{s'})\right)\right]}{\sqrt{E\left[ \left(d(X_s, \mu_s) - Ed(X_s, \mu_s)\right)^2\right] E\left[\left(d(X_{s'}, \mu_{s'})-Ed(X_{s'}, \mu_{s'})\right)^2\right]}},  \label{eq:cor_def2} \\
    \hat{\rho}_{ss'}' &= \frac{ \frac{1}{n}\sum_{i=1}^n \left(d(X_{si}, \hat{\mu}_s) - \overline{d}_s\right) \left(d(X_{s'i}, \hat{\mu}_{s'}) - \overline{d}_{s'}\right) }{\sqrt{\frac{1}{n}\sum_{i=1}^n \left(d(X_{si}, \hat{\mu}_s)-\overline{d}_s\right)^2  \frac{1}{n}\sum_{i=1}^n \left(d(X_{s'i}, \hat{\mu}_{s'}) - \overline{d}_{s'}\right)^2}}. \label{eq:cor_def2_samp}
\end{align}
The non-centered correlation lies in $[0,1]$, and the centered version is in $[-1,1]$.

To illustrate the behavior of each definition of correlation, we simulated data from two metric spaces, both Euclidean, and each equipped with the $L^1$ metric.
The first random variable is $X_1\in \Omega_1 = \mathbb{R}$ and the second random variable is $X_2 \in \Omega_2 = \mathbb{R}$.
In a Euclidean sense, we are sampling bivariate data in $\mathbb{R}^2$.
We simulate data under 5 different scenarios.
For all scenarios the $X_1$ values are generated from  a $\text{Uniform}(0,1)$ distribution.
The $X_2$ values are generated either as a noisy function of $X_1$ (scenarios 1-4) or from a $\text{Uniform}(0,1)$ distribution independent of $X_1$ (scenario 5).
For each scenario, we compute both the non-centered and centered sample Fr\'echet correlation, treating each bivariate observation as an object vector, $(X_{1i}, X_{2i})$.
Figure \ref{fig:cor_examples} shows a scatter plot of observations for each scenario,  along with the non-centered and centered Fr\'echet correlations of each.
\begin{figure}
    \centering
    \includegraphics[width=0.98\textwidth]{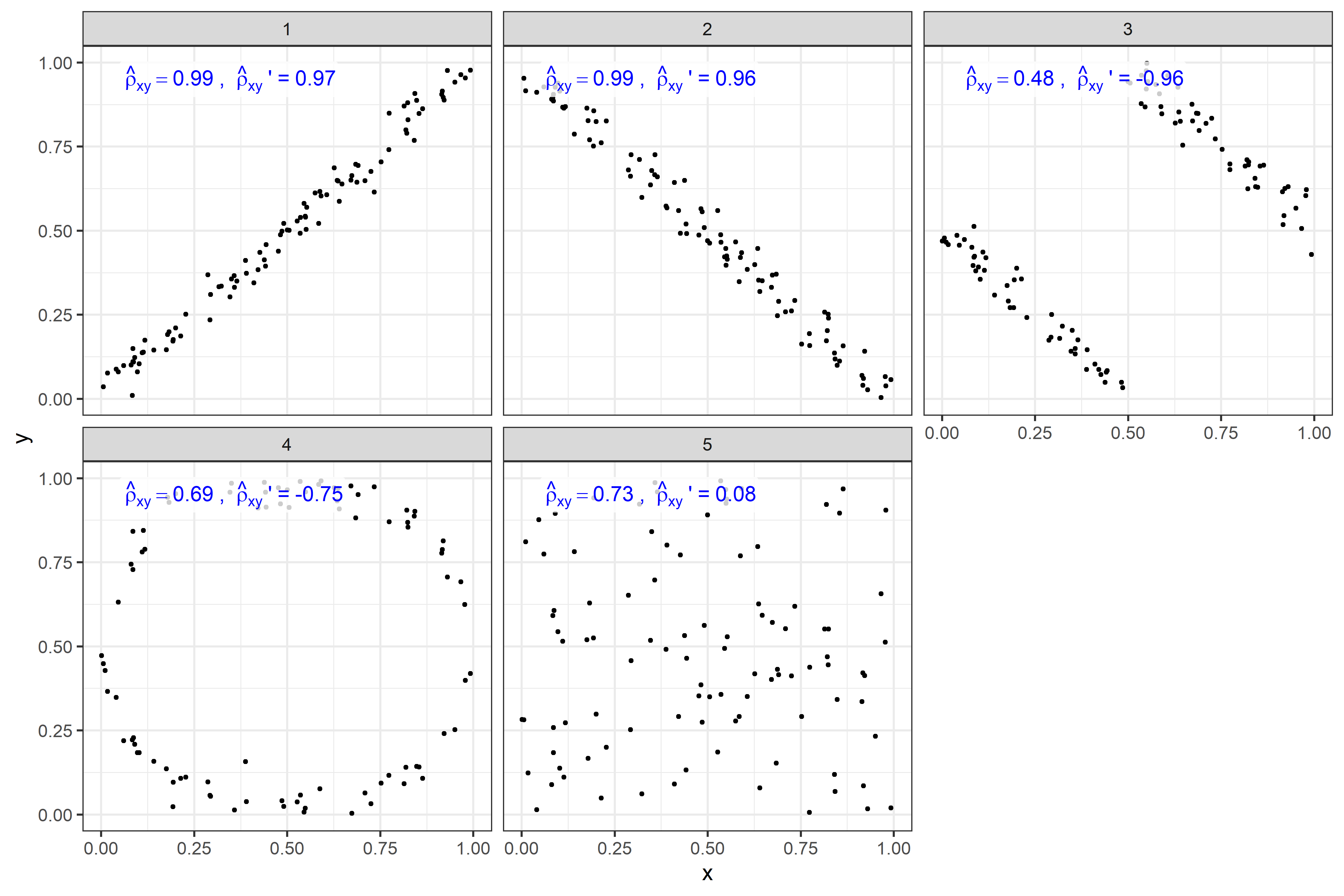}
    \caption{Non-centered ($\hat{\rho}_{12}$) and centered ($\hat{\rho}_{12}'$) Fr\'echet correlation for 5 Euclidean bivariate scenarios. Scenarios 1 and 2 illustrate how both definitions show strong positive correlation when Euclidean correlation is either strongly positive or strongly negative. Scenarios 3 and 4 highlight how the centered Fr\'echet correlation is negative when $X_1$ and $X_2$ tend to move in opposite directions relative to their means. Scenario 5 shows that the centered Fr\'echet correlation near zero when $X_1$ and $X_2$ are independent.}
    \label{fig:cor_examples}
\end{figure}
In terms of Euclidean correlation, scenarios 1 and 2 show strong positive and negative correlation respectively, scenario 3 shows weaker positive correlation, and scenarios 4 and 5 have correlation near zero.

The non-centered Fr\'echet correlation measures how much the $X_1$ and $X_2$ deviate from their mean at the same time. 
In this regard, scenarios 1 and 2 both exhibit strong positive non-centered Fr\'echet correlation, as the points deviate from $(0.5, 0.5)$ together.
In scenarios 3 and 4, $X_1$ values far from their mean are paired with $X_2$ values near their mean, and vice versa, resulting in smaller (but still positive) non-centered Fr\'echet correlations.
Finally, scenario 5 shows how the non-centered Fr\'echet correlation disagrees with the classical Euclidean definition.

The centered Fr\'echet correlation focuses on the degree to which the distance from $X_1$ to its mean is above or below average, and similarly for $X_2$.
For scenarios 1 and 2, $X_1$ and $X_2$ are both above and below average distance simultaneously, leading to strong positive correlation.
In scenario 3, $X_1$ is near its mean when $X_2$ is far away from its and vice versa, leading to a strong negative centered Fr\'echet correlation, similarly for scenario 4.
In scenario 5, when $X_1$ and $X_2$ are generated independently, the centered correlation is near zero.
Scenarios 3-5 in particular show how the centered version of correlation retains some desirable characteristics of Euclidean correlation and is arguably easier to interpret than the non-centered version.

The examples in Figure \ref{fig:cor_examples} illustrate how both definitions of Fr\'echet correlation differ from the Euclidean notion when looking at Euclidean data, but this is expected since metric spaces force us to redefine what it means to be correlated using only distances.
In the rest of this paper, we will consider the non-centered version of Fr\'echet covariance, as it integrates nicely with Dubey \& M\"uller's Fr\'echet variance.

\section{Properties of Fr\'echet Covariance}

In this section we define the (non-centered) Fr\'echet covariance matrix, and establish consistency for the (non-centered) Fr\'echet covariance defined in \eqref{eq:cov_def1_samp}.
We begin by considering the Fr\'echet covariance matrix constructed using Dubey \& M\"uller's Fr\'echet variances on the diagonal and our non-centered Fr\'echet covariance on the off-diagonals.
Given $S$ metric spaces $\{\Omega_{s}\}$ for $s \in \{1, \ldots, S\}$, and random object vector $\textbf{X} = ( X_1, \ldots, X_S ) \in \Omega_1 \times \ldots \times \Omega_S$ with probability measure $P: \Omega_1 \times \ldots \times \Omega_S \rightarrow \mathbb{R}_{\geq 0}$ and Fr\'echet mean vector $( \mu_{1}, \ldots, \mu_{S})$.
Define the $(s,s')$ element of the population covariance matrix $\Sigma$ and the sample covariance matrix $\widehat{\Sigma}$ as 
\begin{align}
    \Sigma(s,s') &= \begin{cases} \sigma_s^2 & s = s' \\ \sigma_{ss'} & s \neq s' \end{cases}, \\
    \widehat{\Sigma}(s,s') &= \begin{cases} \widehat{\sigma^2}_s & s = s' \\ \widehat{\sigma}_{ss'} & s \neq s' \end{cases}.
\end{align}
Defining the vector of distances $d_\mu(\textbf{X}) = \left( d(X_1, \mu_1)\allowbreak , \ldots, \allowbreak d(X_S, \mu_S) \right)^T$, we can write $\Sigma = E[d_\mu d_\mu^T]$.
Similarly, if $X_{si}$ is the $i$th observation in the $s$th metric space and $\hat{\mu}_s$ the sample Fr\'echet mean in the $s$th metric space, define $d_{\hat{\mu}i} = \allowbreak \left( d(X_{1i}, \hat{\mu}_{1}), \ldots, d(X_{si}, \hat{\mu}_{s}) \right)^T$ so that $\widehat{\Sigma} = \frac{1}{n} \sum_{i=1}^n d_{\hat{\mu}i} d_{\hat{\mu}i}^T$.
This makes it clear the both $\Sigma$ and $\widehat{\Sigma}$ are positive semi-definite:

\begin{lemma}
    The Fr\'echet population covariance matrix $\Sigma$ and sample covariance matrix $\widehat{\Sigma}$ are both positive semi-definite. 
    If the elements of $d_\mu$ are non-degenerate, then $\Sigma$ is positive definite.
    If the vectors $\{ d_{\hat{\mu}1}, \ldots, d_{\hat{\mu}n} \}$ span $\mathbb{R}^S$, then $\widehat{\Sigma}$ is positive definite.
    \label{lem:cov_mat_psd}
\end{lemma}

\begin{proof}[Proof of Lemma \ref{lem:cov_mat_psd}]
    Given vector $\mathbf{z} \in \mathbb{R}^S$, then 
    \begin{align}
        z^T \Sigma z &= z^T E[d_\mu d_\mu^T] z
        = E[z^T d_\mu d_\mu^T z] 
        = E[||d_\mu^T z||^2] \geq 0.
    \end{align}
    The quantity is strictly greater than zero if the elements of $d_\mu$ are non-degenerate.
    Similarly for $\widehat{\Sigma}$,
    \begin{align}
        z^T \widehat{\Sigma} z &= \frac{1}{n} \sum_{i=1}^n z^T d_{\hat{\mu}i} d_{\hat{\mu}i}^T z 
        = \frac{1}{n} \sum_{i=1}^n ||d_{\hat{\mu}i}^T z||^2 \geq 0.
        \label{eq:lemma1}
    \end{align}
    If the vectors $\{ d_{\hat{\mu}1}, \ldots, d_{\hat{\mu}n} \}$ span $\mathbb{R}^D$, then any vector $z$ can be written as $a_1 d_{\hat{\mu}1} + \cdots + a_n d_{\hat{\mu}n}$ so that $z^T z = a_1 d_{\hat{\mu}1}^T z + \cdots + a_n d_{\hat{\mu}n}^T z$.
    Hence if \eqref{eq:lemma1} is equal to zero, then $d_{\hat{\mu}i}^T z = 0$ for each $i$, implying $z^T z = 0$, so $z = 0$.
\end{proof}

From now on we assume that all population and sample covariances are strictly positive definite.
The remaining results of this section will focus on the the covariance of $(X_s, X_{s'}) \in \Omega_s \times \Omega_{s'}$, and require the following assumptions:

\begin{enumerate}
    \item [A.1] The population and sample Fr\'echet means in each metric space, $(\mu_s, \mu_{s'})$ and $(\hat{\mu}_s, \hat{\mu}_{s'})$, exist and are unique, and for any $\epsilon > 0$, $\inf_{d(\omega, \mu_s) > \epsilon} E[d^2(\omega, X_s)] > E[d^2(\mu_s, X_s)]$ and $\inf_{d(\omega, \mu_{s'}) > \epsilon} E[d^2(\omega, X_{s'})] > E[d^2(\mu_{s'}, X_{s'})]$. In other words, in each metric space, the expected distance from the mean is uniquely minimum compared to the expected distance from any other point $\omega$ which is further than $\epsilon$ from the mean.
\end{enumerate}

This assumption is analogous to one made by Dubey \& M\"uller in \cite{dubey_frechet_2019}, and Petersen \& M\"uller show in \cite{petersen_frechet_2019} that it is satisfied by the following two metric spaces: univariate probability distributions with compact support and finite second moment equipped with the $L^2$-Wasserstein metric, the set of correlation matrices of a fixed dimension equipped with the Frobenius metric.
Dubey \& M\"uller additionally claim that the set of graph Laplacians of connected, undirected, simple graphs of a fixed dimension also satisfy the assumptions~\cite{dubey_frechet_2019}.
Under Assumptions A.1, we establish consistency of the sample Fr\'echet covariance $\widehat{\sigma}_{ss'}$.

\begin{theorem}[Consistency]
    Under Assumption 1 above, the sample covariance converges to the population covariance in probability, that is $\widehat{\sigma}_{ss'} \overset{p}\rightarrow \sigma_{ss'}$ as $n\rightarrow\infty$.
    \label{thm:cov_consistency}
\end{theorem}
The proof of Theorem \ref{thm:cov_consistency} is provided in Appendix.

\section{Tests for Differences in Means and Variance / Covariance Matrices}

Our goal is to investigate various tests for differences between $J$ groups in $S$ metric spaces, as an analog to the MANOVA procedure.
We will consider various tests based on several test statistics, but in all cases the null hypothesis we wish to test is simultaneous equality of the mean vector across groups and equality of covariance matrix across groups.
Letting $\bm{\mu}_{j} = (\mu_{j1}, \ldots, \mu_{jS})^T$ and $\Sigma_j$ be the mean vector and covariance matrix for group $j$ respectively, then the null hypothesis can be expressed as
\begin{align}
    \text{H}_0:\quad \bm{\mu}_1 = \ldots = \bm{\mu}_J,\quad \Sigma_1 = \ldots = \Sigma_J.
    \label{eq:sms_null}
\end{align}
If the groups deviate from the null hypothesis, we expect the \textit{sample} means and \textit{sample} covariance matrices to deviate from one another, so we seek test statistics which quantify these differences.

\subsection{Adaptations of Classical MANOVA Statistics}

Several Euclidean MANOVA tests, including Wilks' lambda~\cite{wilks_certain_1932}, Pillai \& Bartlett's trace~\cite{pillai_new_1955}, Lawley \& Hotelling's trace~\cite{hotelling_generalized_1951}, and Roy's Root~\cite{roy_aspects_1958}, test for differences in means by relying on a partition of the total sums of squares and cross products matrix into a between group comoponent (the treatment) and a within group component (the residuals).
Such a partition is possible in Euclidean space, thanks to the Pythagorean theorem, but not always possible in generic metric spaces.
We can circumvent this problem, however, following the approach of Dubey \& M\"uller's $F_s$ statistic, by comparing the pooled covariance matrix to the weighted mean of group covariance matrices.

Let $\widehat{\sigma}_{jss'}$ be the covariance of group $j$, and $\widehat{\sigma}_{pss'}$ be the pooled covariance.
Define the $(s,s')$ element of the population and sample pooled covariance matrices as 
\begin{align}
    \Sigma_{p}(s,s') &= \begin{cases} \sigma_{ps}^2 & s = s' \\ \sigma_{pss'} & s \neq s' \end{cases}, \\
    \widehat{\Sigma}_{p}(s,s') &= \begin{cases} \widehat{\sigma}_{ps}^2 & s = s' \\ \widehat{\sigma}_{pss'} & s \neq s' \end{cases}.
\end{align}
Similarly, define the $(s,s')$ element of the population and sample group weighted mean covariance matrices as 
\begin{align}
    \Sigma_{g}(s,s') &= \begin{cases} \sum_{j=1}^J \gamma_j \sigma^2_{js} & s = s' \\ \sum_{j=1}^J \gamma_j \sigma_{jss'} & s \neq s' \end{cases}, \\
    \widehat{\Sigma}_{g}(s,s') &= \begin{cases} \sum_{j=1}^J \gamma_j \widehat{\sigma}^2_{js} & s = s' \\ \sum_{j=1}^J \gamma_j \widehat{\sigma}_{jss'} & s \neq s' \end{cases}.
\end{align}
Whereas $\Sigma_p$ and $\widehat\Sigma_p$ capture the total variation across all groups, $\Sigma_g$ and $\widehat\Sigma_g$ capture the within group variation.
When all groups share the same mean vector and covariance matrix, then $\Sigma_p = \Sigma_g$, so $\widehat\Sigma_p$ will be close to $\widehat\Sigma_g$.
As the group mean vectors deviate from one another, $\widehat\Sigma_p$ will inflate while $\widehat\Sigma_g$ will remain more or less the same.

It turns out that we can mimic Euclidean MANOVA statistics using $\widehat\Sigma_p$ and $\widehat\Sigma_g$, without having an explicit residuals term.
For instance, the Pillai-Bartlett trace~\cite{pillai_new_1955} can be expressed as the trace of $H (H+E)$, where $H$ is the between-group sums of squares and cross product matrix, and E is the residual sums of squares and cross product matrix.
Consider starting with the pooled covariance, $\widehat\Sigma_p$, and subtracting the mean within group covariance, $\widehat\Sigma_g$.
This is similar to the idea of subtracting the within group sum of squares from the total sum of squares, leaving only a between group effect.
Therefore we may loosely identify $\widehat\Sigma_p - \widehat\Sigma_g$ with $H$, and similarly loosely identify $\widehat\Sigma_g$ with $E$, ignoring obvious normalizing constants.
Translating the Pillai-Bartlett trace into our context then produces:
\begin{align}
    \Lambda_{pillai} &= \text{tr}[H (H+E)^{-1}] 
        = \text{tr}\left[ \left(\widehat\Sigma_p - \widehat\Sigma_g\right) \widehat\Sigma_p^{-1} \right] \\
        &= \text{tr}(I - \widehat\Sigma_g \widehat\Sigma_p^{-1}) 
        = \sum_i \lambda_i\left(I - \widehat\Sigma_g \widehat\Sigma_p^{-1}\right),
\end{align}
Where $\lambda_i(A)$ denotes the $i^{th}$ eigenvalue of the matrix $A$.
Under the null hypothesis, $\widehat\Sigma_g \widehat\Sigma_p^{-1}$ is close to the identity matrix $I$, and $\Lambda_{pillai}$ will be close to zero.

Another possible adaptation of a classical MANOVA test would be to perform a Euclidean MANOVA on distances between observations and sample means.
That is, for group $j$ and observation $i$, define the vector of distances $d_{\hat{\mu}ji} = \allowbreak ( d(X_{j1i}, \hat{\mu}_{j1}), \ldots, \allowbreak d(X_{jSi}, \hat{\mu}_{jS}) )^T$.
each vector summarizes an observation based on its distance to its respective group sample mean.
We can then apply the Pillai-Bartlett trace directly to the distances to compare the groups.
We call this statistic $\Lambda_{pillai,d}$.
Clearly, if any of the group variances changes, the average distance must be changing as well, which should be detected by $\Lambda_{pillai,d}$.

\subsection{Statistics Based on Riemannian Geometry}

As was mentioned above, as the group mean vectors deviate from one another, $\widehat\Sigma_p$ will inflate while $\widehat\Sigma_g$ will remain more or less the same, so a natural test would be to compare them to one another.

Since $\widehat\Sigma_g$ is the sum of positive-definite matrices, it is positive definite, whereas $\widehat\Sigma_p$ is positive definite by Lemma \ref{lem:cov_mat_psd}.
The space of symmetric positive definite (SPD) matrices is a manifold, so $\hat\Sigma_p$ and $\hat\Sigma_g$ can be thought of as points on the manifold.
This manifold can be endowed with a Riemannian metric which defines the distance between points in any \textit{tangent space} incident to the manifold.
To avoid confusion, note that the Riemannian metric does not measure distances between points on the manifold (i.e. SPD matrices) directly, but the Riemannian metric can be integrated along a path on the manifold to determine the length of the path, and the distance between two points on the manifold can be taken as the length of the shortest path between them.
For more details on manifolds and Riemannian metrics, refer to the Appendix.
The shortest path length between two SPD matrices provides a means to compare $\widehat\Sigma_g$ to $\widehat\Sigma_p$ and ultimately test for differences between the groups.

We define three more statistics statistics which compare the within group variation $\widehat\Sigma_g$ to the total variation $\widehat\Sigma_p$ using different choices of Riemannian metric:  Euclidean (Euc), affine-invariant (AIRM), and log-euclidean (LERM)\cite{pennec_riemannian_2006, arsigny_geometric_2007, you_re-visiting_2021, moakher_differential_2005}:
\begin{align}
    R_{\mu, Euc} &= d_{Euc}\left(\widehat{\Sigma}_p, \widehat{\Sigma}_g\right) = \sqrt{\sum_i \lambda_i^2 \left(\widehat\Sigma_p - \widehat\Sigma_g \right)}\label{eq:EUC_means} \\
    R_{\mu, AIRM} &= d_{AIRM}\left(\widehat{\Sigma}_p, \widehat{\Sigma}_g\right) = \sqrt{\sum_i \text{log}^2 \left[\lambda_i( \widehat\Sigma_p^{-1} \widehat\Sigma_g )\right]} \label{eq:AIRM_means} \\
    R_{\mu, LERM} &= d_{LERM}\left(\widehat{\Sigma}_p, \widehat{\Sigma}_g\right) = \sqrt{\sum_i \lambda_i^2 \left(\text{Log}(\widehat\Sigma_p) - \text{Log}(\widehat\Sigma_g) \right)} \label{eq:LERM_means}
\end{align}
All of these statistics are expected to be small if all groups have the same population mean vector and will increase as the mean vectors deviate from each other.
The $d_{LERM}$ and $d_{AIRM}$ distances are particularly sensitive to changes in determinant, which will certainly change for $\Sigma_p$ as the mean vectors grow further apart, so we anticipate that $R_{\mu,LERM}$ and $R_{\mu,AIRM}$ will be particularly sensitive to departures from the null hypothesis in \eqref{eq:sms_null}.
Later we will integrate each of the statistics from this section into a global test for differences in means or covariance matrices.

The distances $d_{Euc}$, $d_{AIRM}$ and $d_{LERM}$ defined respectively in \eqref{eq:EUC_metric}, \eqref{eq:AIRM_metric}, and \eqref{eq:LERM_metric} can also be used to compare the group covariance matrices in a pairwise fashion, in order to determine if any covariance matrix differs from the others.
Since the sample covariance matrices are each SPD, the use of these Riemannian metric based distances is appropriate.
This is similar in spirit to Dubey \& M\"uller's $U_s$ statistic, which performs pairwise comparisons of group Fr\'echet variances.
Like $U_s$, we consider the weighted average of all pairwise distances between group covariance matrices, where the weights are the product of the proportions of observations in the respective groups:
\begin{align}
    R_{\Sigma, Euc} &= \sum_{j<j'} \gamma_j \gamma_{j'} d_{Euc}\left(\widehat{\Sigma}_j, \widehat{\Sigma}_{j'}\right) &= \sum_{j<j'} \gamma_j \gamma_{j'} \sqrt{\sum_i \lambda_i^2 \left(\widehat\Sigma_j - \widehat\Sigma_{j'} \right)}\label{eq:EUC_vars} \\
    R_{\Sigma, AIRM} &= \sum_{j<j'} \gamma_j \gamma_{j'} d_{AIRM}\left(\widehat{\Sigma}_j, \widehat{\Sigma}_{j'}\right) &= \sum_{j<j'} \gamma_j \gamma_{j'} \sqrt{\sum_i \text{log}^2 \left[\lambda_i( \widehat\Sigma_j^{-1} \widehat\Sigma_{j'} )\right]} \label{eq:AIRM_vars} \\
    R_{\Sigma, LERM} &= \sum_{j<j'} \gamma_j \gamma_{j'} d_{LERM}\left(\widehat{\Sigma}_j, \widehat{\Sigma}_{j'}\right) &= \sum_{j<j'} \gamma_j \gamma_{j'} \sqrt{\sum_i \lambda_i^2 \left(\text{Log}(\widehat\Sigma_j) - \text{Log}(\widehat\Sigma_{j'}) \right)} \label{eq:LERM_vars}
\end{align}
Each of these quantities is expected to be small when the population covariance matrices of all groups are equal, and grow as the covariance matrices differ from each other.
Since each type of distance is rooted in a different underlying geometry, covariance matrices which appear close in one geometry may appear further apart in another geometry, and vice versa.
Therefore we expect the three statistics introduced in this section to have relative strengths and weaknesses for the types of departures from equality of group covariance matrices they can detect well.
We anticipate that the statistics using $d_{AIRM}$ and $d_{LERM}$ will be more sensitive to differences in matrix determinant, while the statistic using $d_{Euc}$ may be better equipped to detect differences between matrices of similar determinants.

Comparing the covariance matrices between groups implicitly compares the metric space to metric space dependence of the random object vectors across groups.
Another way to compare dependence between groups is to compare their centered Fr\'echet correlations.
Define the $(s,s')$ element of the centered sample correlation matrix $\hat{P}_j$ of group $j$ be defined as
\begin{align}
    \hat{P}_{j}(s,s') &= \begin{cases} 1 & s = s' \\ \hat{\rho}_{jss'}' & s \neq s' \end{cases},
\end{align}
where  $\hat{\rho}_{jss'}'$ is the sample centered covariance between metric spaces $s$ and $s'$ in group $j$.
Since the centered Fr\'echet correlation is just the classical Euclidean correlation applied to the distances of the random objects from their means, $\hat{P}_{j}$ is necessarily positive semi-definite, so we can use the same approach to compare correlation matrices as we used to compare covariance matrices above:
\begin{align}
    R_{P, Euc} &= \sum_{j<j'} \gamma_j \gamma_{j'} d_{Euc}\left(\widehat{P}_j, \widehat{P}_{j'}\right) &= \sum_{j<j'} \gamma_j \gamma_{j'} \sqrt{\sum_i \lambda_i^2 \left(\widehat{P}_j - \widehat{P}_{j'} \right)}\label{eq:EUC_cors} \\
    R_{P, AIRM} &= \sum_{j<j'} \gamma_j \gamma_{j'} d_{AIRM}\left(\widehat{P}_j, \widehat{P}_{j'}\right) &= \sum_{j<j'} \gamma_j \gamma_{j'} \sqrt{\sum_i \text{log}^2 \left[\lambda_i( \widehat{P}_j^{-1} \widehat{P}_{j'} )\right]} \label{eq:AIRM_cors} \\
    R_{P, LERM} &= \sum_{j<j'} \gamma_j \gamma_{j'} d_{LERM}\left(\widehat{P}_j, \widehat{P}_{j'}\right) &= \sum_{j<j'} \gamma_j \gamma_{j'} \sqrt{\sum_i \lambda_i^2 \left(\text{Log}(\widehat{P}_j) - \text{Log}(\widehat{P}_{j'}) \right)} \label{eq:LERM_cors}
\end{align}
It is possible that comparing correlations directly will be more sensitive to changes in dependence than comparing covariance matrices, so the correlation based statistics are worth consideration.

\subsection{Statistics Based on the Fr\'echet ANOVA}

Just as Dubey and M\"uller defined the $T_s$ statistic to compare group means and variances, we may analogously define the $T_{ss'}$ statistic as
\begin{align}
    T_{ss'} = \frac{nU_{ss'}}{\sum_{j=1}^J \frac{\gamma_j}{\widehat{\text{Var}}(\widehat{\sigma}_{jss'})}} + \frac{nF_{ss'}^2}{\sum_{j=1}^J \gamma_j^2 \widehat{\text{Var}}(\widehat{\sigma}_{jss'})}\label{eq:frechet_anova_cov},
\end{align}
where 
\begin{align}
    F_{ss'} &= \widehat{\sigma}_{pss'} - \sum_{j=1}^J \gamma_j \widehat{\sigma}_{jss'} \\
    U_{ss'} &= \sum_{j < j'} \frac{\gamma_j \gamma_{j'}}{\widehat{\text{Var}}(\widehat{\sigma}_{jss'}) \widehat{\text{Var}}(\widehat{\sigma}_{j'ss'})} \left( \widehat{\sigma}_{jss'} - \widehat{\sigma}_{j'ss'} \right)^2,
\end{align}
Unlike $T_s$, we do not have asymptotic results for $T_{ss'}$, though if the Fr\'echet covariance is indeed asymptotically normal, then under the null hypothesis $T_{ss'}$ would be asymptotically $\chi_{J-1}^2$ as $n\rightarrow \infty$, the proof being identical to the proof of Theorem 2 in \cite{dubey_frechet_2019}.
Based on simulation studies, we suspect asymptotic normality of the Fr\'echet covariance may hold at least in some circumstances, suggesting we may be justified in comparing $T_{ss'}$ to a $\chi_{J-1}^2$ distribution to test the null hypothesis.
We explore the validity of this approach in a few example scenarios below.

In the preceding three sections we introduced several statistics, two inspired by the classical MANOVA test, three comparing pooled and group mean covariance matrices, three comparing group covariance matrices, three comparing group centered correlation matrices, and one analogous to the statistic involved in the Fr\'echet ANOVA.
In the next section we will employ these statistics in the construction of tests for the null hypothesis in \eqref{eq:sms_null}.

\subsection{Testing the Null Hypothesis}

Broadly speaking, we will consider three main approaches to testing the null hypothesis in \eqref{eq:sms_null}.
The first approach will combine the Riemannian mean, covariance matrix, and correlation matrix based statistics from the previous two sections into a global permutation-based test, achieving the desired Type I error rate using a Bonferroni correction.
Doing this for each Riemannian metric gives three global tests.
The second approach will combine several Fr\'echet ANOVA test statistics, each of which applies to just one or two metric spaces, into a single test, also using a Bonferroni correction to account for multiple tests.
Using the $T_s$ and $T_{ss'}$ statistics, we will consider a test which compares against the $\chi_{J-1}^2$ distribution, as well as a permutation test.
The third approach will apply both extensions of the Pillai-Bartlett trace statistic discussed above, resulting in two tests, one based on permutations and one based on an approximation using the F distribution. Each test will be discussed in more detail below.

In the absence of any distributional assumptions on the underlying metric space probability measures, we lack distributions for the Riemannian statistics under the null hypothesis.
The null hypothesis only assumes equality of the first and second moments between the groups, however if we make the stronger assumption that all moments are equal, then observations are exchangable between groups, and we may generate an approximate sampling distribution for the statistics using either label permutation or bootstrap.
Permutation tests have recently been employed to compare groups of persistence diagrams~\cite{berry_functional_2020}, another non-euclidean setting, and Dubey \& M\"uller~\cite{dubey_frechet_2019} found in simulation studies that a bootstrapped version of their $T_s$ statistic showed improved power over the version comparing against a $\chi_{J-1}^2$ distribution.

We create three global tests, one for each Riemannian-based metric, each comprising of a mean test, a covariance matrix test, and a correlation matrix test, then apply a Bonferroni correction by rejecting if any statistic is greater than the $1-\alpha/3$ upper-quantile of its permutation distribution.
Explicitly, the three tests are:
\begin{enumerate}
    \item $\bm{R_{Euc}}$: reject if the permutation test for any of $R_{\mu, Euc}$, $R_{\Sigma, Euc}$, $R_{P, Euc}$ rejects at the $\alpha/3$ level
    \item $\bm{R_{AIRM}}$: reject if the permutation test for any of $R_{\mu, AIRM}$, $R_{\Sigma, AIRM}$, $R_{P, AIRM}$ rejects at the $\alpha/3$ level
    \item $\bm{R_{LERM}}$: reject if the permutation test for any of $R_{\mu, LERM}$, $R_{\Sigma, LERM}$, $R_{P, LERM}$ rejects at the $\alpha/3$ level
\end{enumerate}

Another way to test the null hypothesis is by applying the Fr\'echet ANOVA in \eqref{eq:frechet_anova} to each of the S metric spaces separately, and also apply the analogous ANOVA for covariances in \eqref{eq:frechet_anova_cov} for each distinct pairing of metric spaces, performing $S(S+1)/2$ tests in total.
Again we can use a Bonferroni correction to construct a single global test, $\bm{T_{FA}}$, which we reject if any of the $S$ ANOVA tests or any of the $S(S-1)/2$ ANOVA for covariance tests reject at the $2\alpha/(S(S+1))$ significance level.
We compare each statistic against the $1-2\alpha/(S(S+1))$ upper quantile of a $\chi_{J-1}^2$ distribution.
We also consider permutation based tests, denoted $\bm{T_{FA, perm}}$, where each statistic is compared against the $1-2\alpha/(S(S+1))$ upper quantile of the permutation distribution, to investigate if it produces higher power at lower sample sizes as was observed in \cite{dubey_frechet_2019}.

Finally, we construct tests using each version of the Pillai-Bartlett trace, $\Lambda_{pillai}$ and $\Lambda_{pillai,d}$ to examine a more straightforward application of existing MANOVA methods.
Again, we lack distributional results for $\Lambda_{pillai}$, so we resort to a permutation test comparing $\Lambda_{pillai}$ to the upper $1-\alpha$ quantile of its permutation distribution.
On the other hand, $\Lambda_{pillai,d}$ is just the Euclidean Pillai-Bartlett trace, which typically assumes normality and homogeneity of variance.
The statistic is usually transformed and compared to an $F$ distribution, and tends to be robust against departure from normality assumptions, see \cite{anderson_introduction_1984}  or ~\cite{krzanowski_multivariate_1994} for more details.
We'll denote these tests as $\bm{\Lambda_{pillai}}$ and $\bm{\Lambda_{pillai,d}}$ respectively.

In total, we have seven tests reflecting three broad approaches to comparing two or more groups in multiple metric spaces:
$\bm{R_{Euc}}$, Riemannian permutation test based on the Euclidean Riemannian metric; $\bm{R_{AIRM}}$, Riemannian permutation test based on the affine invariant Riemannian metric; $\bm{R_{LERM}}$, Riemannian permutation test based on the log-Euclidean Riemannian metric; $\bm{T_{FA}}$, asymptotic test based on the Fr\'echet ANOVA (recall that asymptotic results for the Fr\'echet covariance remain to be proven); $\bm{T_{FA,perm}}$, permutation test based on the Fr\'echet ANOVA; $\bm{\Lambda_{pillai}}$, permutation test adapted from the Pillai-Bartlett trace; $\bm{\Lambda_{pillai,d}}$, classical Pillai-Bartlett trace applied to the distances of observations to their respective group means.
In the next section, we will evaluate these statistics under various departures from the null hypothesis in two separate scenarios.

\section{Simulation Studies}

To compare the power and Type I error rate of the tests described in the previous section, we simulate data under the null hypothesis and under various departures from the null hypothesis for two different scenarios, performing each test on each sample.
All permutation-based tests permuted group labels 500 times to generate a null distribution, and 6,000 simulated data sets were used to estimate the power and Type I error rate.
In both scenarios, we use a desired Type I error rate of $\alpha = 0.05$.

\subsection{Scenario 1: Normal Distributions in Two Metric Spaces, Two Groups}
In the first scenario, we compare $J=2$ groups in $S=2$ metric spaces, $\Omega_1$ and $\Omega_2$.
Random objects $X_1 \in \Omega_1$ are normal distributions with random mean $a$ and unit variance, and random objects $X_2 \in \Omega_2$ are normal distributions with random mean $b$ and unit variance.
Both $\Omega_1$ and $\Omega_2$ are equipped with the Wasserstein-2 metric which reduces to the $L_2$ Euclidean distance between the means of the probability distributions. 
The Fr\'echet mean of $X_1$ is a $N(E(a), 1)$ distribution, and the mean of $X_2$ is a $N(E(b), 1)$ distribution.
The Fr\'echet variance of $X_1$ is simply $var(a)$, and the Fr\'echet variance of $X_2$ is $var(b)$.
For this scenario, we performed five simulation studies to examine different ways the data might depart from the null hypothesis, summarized in Table \ref{tab:scen1_studies}.
In each study, both groups have sample size $n = 100$.

The first study investigates a change in the Fr\'echet mean of $X_1$ for group 2.
That is, $a\sim N(0, 0.5^2)$ and $b\sim N(0,0.2^2)$ for group 1, and $a\sim N(\delta, 0.5^2)$ and $b\sim N(0,0.2^2)$ for group 2, where we let $\delta$ range between $-1$ and $1$.
Note that $0$ corresponds to the null hypothesis, and $a$ and $b$ are independent.
\begin{figure}[h!]
    \centering
    \begin{subfigure}{0.48\textwidth}
        \centering
        \includegraphics[width=\textwidth]{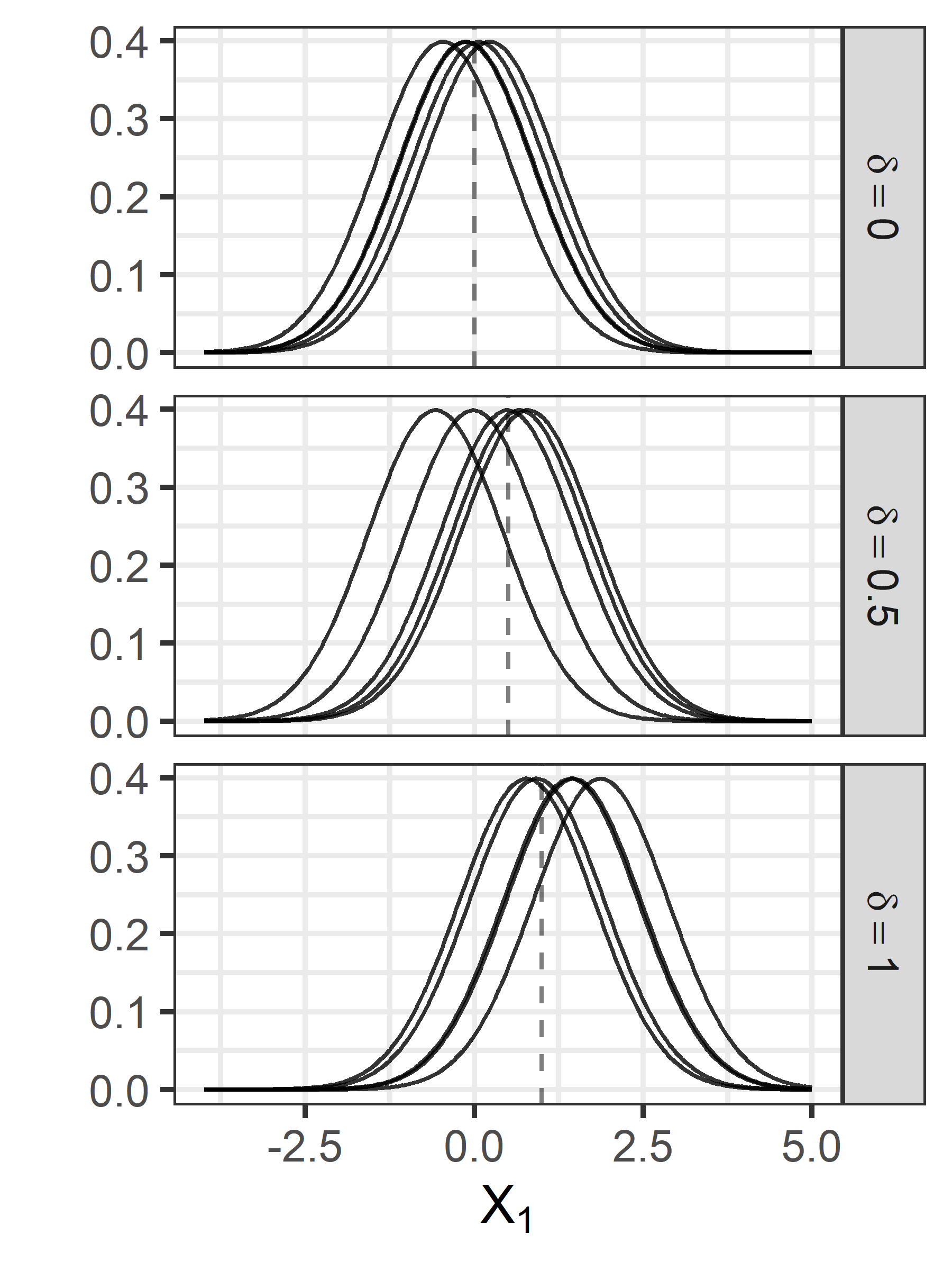}
        \caption{Five realizations of $X_1$ for group 2 for simulation scenario 1, study 1, for three values of $\delta$, which is the mean of $a$.}
        \label{fig:sms_scen1_data_a}
    \end{subfigure}\hfill
    \begin{subfigure}{0.48\textwidth}
        \centering
        \includegraphics[width=\textwidth]{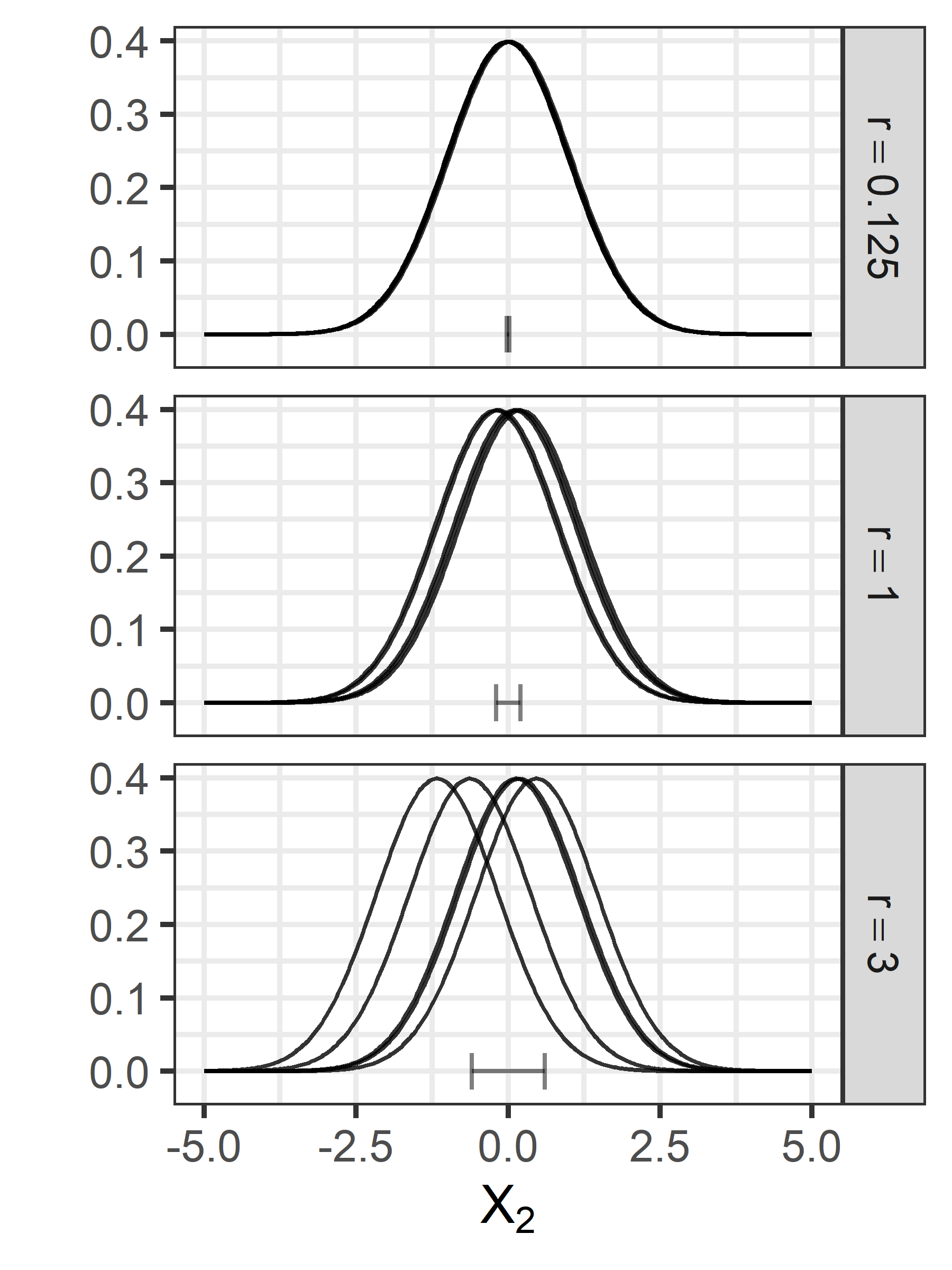}
        \caption{Five realizations of $X_1$ for group 2 in simulation scenario 1, study 2, for three values of $r$. The length of the whiskers in each plot shows the standard deviation of $b$, which is $0.2r$.}
        \label{fig:sms_scen1_data_b}
    \end{subfigure}
    \begin{subfigure}{1.0\textwidth}
        \centering
        \includegraphics[width=\textwidth]{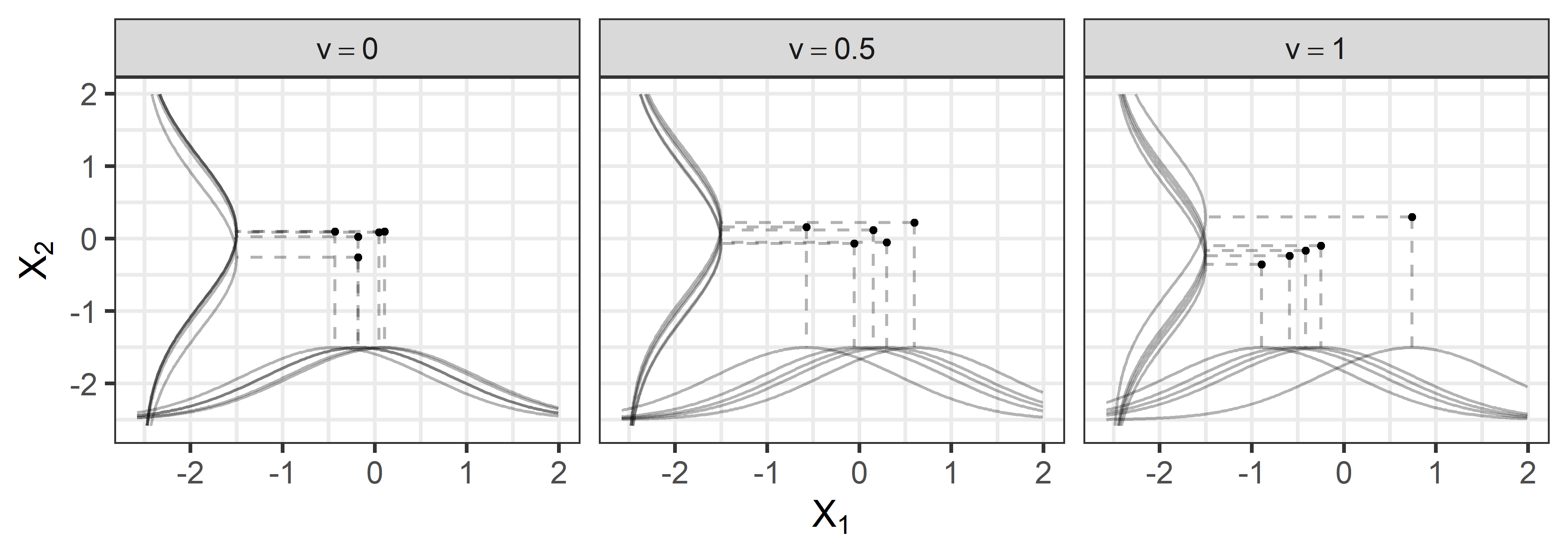}
        \caption{Five realizations of $(X_1,X_2)$ for group 2 in simulation scenario 1, study 3, for three values of $v = \text{cor}(a,b)$. $X_1$ and $X_2$ are illustrated in the margins, and the points correspond to $(a,b)$. Dashed lines connect each instance of $(X_1,X_2)$ to its corresponding point $(a,b)$.}
        \label{fig:sms_scen1_data_c}
    \end{subfigure}
    \caption{Data examples for the first four studies in scenario 1}
    \label{fig:sms_scen1_data}
\end{figure}
Figure \ref{fig:sms_scen1_data_a} shows five realizations of $X_1$ for group 2, under different values of $\delta$.

The second study investigates a change in the variance of $X_2$ for group 2, such that $b\sim N(0, (0.2r)^2)$ with $r$ ranging between 0.125 and 3, with $r=1$ corresponding to the null hypothesis.
Figure \ref{fig:sms_scen1_data_b} shows five realizations of $X_2$ for group 2, under different values of $r$.
The distribution of group 1 is the same as in study 1, and again $a$ and $b$ are independent.

The third and fourth studies investigate changes in the dependence between $X_1$ and $X_2$, where dependence is induced by giving $a$ and $b$ a non-zero covariance.
For study 3, $a$ and $b$ for group 1 remain independent, while $\text{cor}(a, b) = v$ for group 2, with $v$ ranging between 0 to 1, and $v=0$ corresponding to the null hypothesis.
Note that $a$ and $b$ are real numbers, so $v$ is the classic (Euclidean) correlation.
For study 4, we let $a$ and $b$ for group 1 have some dependence (i.e. $\text{cor}(a, b) = \sqrt{0.5}$ ), while we let $\text{cor}(a, b) = v$ range between $0$ and $1$ for group 2.
Figure \ref{fig:sms_scen1_data_c} shows five realizations of $(X_1,X_2)$ in group 2 under different values of $v$, where $X_1$ and $X_2$ are illustrated in the margins, and the points correspond to $(a, b)$ for each realization.

The fifth study investigates multiple simultaneous departures from the null hypothesis.
Here $a\sim N(\delta, 0.5^2)$ and $b\sim N(0, (0.2r)^2)$, and $\text{cor}(a, b) = v$ for group 2, whereas in group 1, $a$ and $b$ have the same distributions as in study 1.
We allow $\delta$ to range between $0$ to $1$, $r$ range from $1$ to $3$, and $v$ range from $0$ to $1$ all simultaneously.
Specifically, we let $(\delta, r, v) = (1-\Delta)\times(0,1,0) +  \Delta\times(1,3,1)$, with $\Delta \in [0,1]$ indicating the overall ``effect size''.

\begin{table}[h!]
    \centering
    \begin{tabular}{l|c|c|c|c}
        \textbf{Study / Group} & $\bm{a}$ & $\bm{b}$ & $\bm{a}$ \textbf{vs.} $\bm{b}$ & \textbf{Parameter Range} \\
        \hline
        Study 1: Means & & & & \\
        \quad - Group 1 & $\bm{N(0, 0.5^2)}$ & $N(0,0.2^2)$ & Independent & \multirow{2}{*}{$-1 \leq \delta \leq 1$} \\
        \quad - Group 2 & $\bm{N(\delta, 0.5^2)}$ & $N(0,0.2^2)$ & Independent & \\
        Study 2: Variances & & & & \\
        \quad - Group 1 & $N(0, 0.5^2)$ & $\bm{N(0,0.2^2)}$ & Independent & \multirow{2}{*}{$0.125 \leq r \leq 3$} \\
        \quad - Group 2 & $N(0, 0.5^2)$ & $\bm{N(0,(0.2r)^2)}$ & Independent & \\
        Study 3: Covariances & & & & \\
        \quad - Group 1 & $N(0, 0.5^2)$ & $N(0,0.2^2)$ & \bf{Independent} & \multirow{2}{*}{$0 \leq v \leq 0.9$} \\
        \quad - Group 2 & $N(0, 0.5^2)$ & $N(0,0.2^2)$ & $\bm{\text{cor}(a, b) = v}$ & \\
        Study 4: Covariances & & & & \\
        \quad - Group 1 & $N(0, 0.5^2)$ & $N(0,0.2^2)$ & $\bm{\text{cor}(a, b) = \sqrt{0.5}}$  & \multirow{2}{*}{$0 \leq v \leq 0.9$} \\
        \quad - Group 2 & $N(0, 0.5^2)$ & $N(0,0.2^2)$ & $\bm{\text{cor}(a, b) = v}$ & \\
        Study 5: Composite & & & & \\
        \quad - Group 1 & $\bm{N(0, 0.5^2)}$ & $\bm{N(0,0.2^2)}$ & \bf{Independent} & \multirow{3}{*}{\makecell{$0 \leq \delta \leq 1$ \\ $1 \leq r \leq 3$ \\  $0 \leq v \leq 0.9$} } \\
        \quad - Group 2 & $\bm{N(\delta, 0.5^2)}$ & $\bm{N(0,(0.2r)^2)}$ & $\bm{\text{cor}(a, b) = v}$ & \\
        & & & & 
    \end{tabular}
    \caption{Summary of simulation parameters for groups 1 and 2 in each study of scenario 1. Bold sections highlight the differences between groups in each study.}
    \label{tab:scen1_studies}
\end{table}
\begin{figure}[h!]
    \centering
    \begin{subfigure}[t]{0.4\textwidth}
        \centering
        \includegraphics[width=\textwidth]{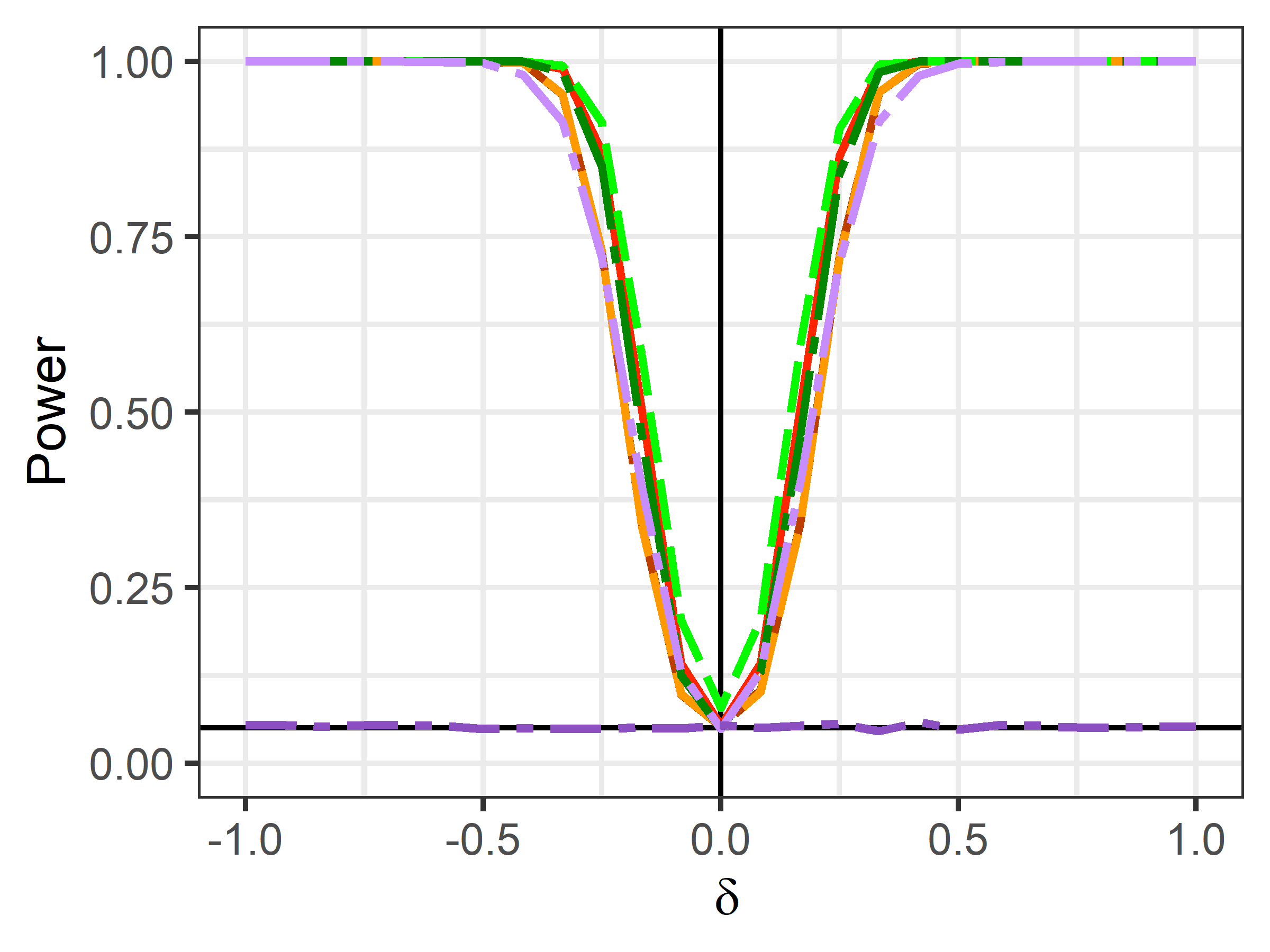}
        \caption{Study 1: Means}
    \end{subfigure}%
    \begin{subfigure}[t]{0.4\textwidth}
        \centering
        \includegraphics[width=\textwidth]{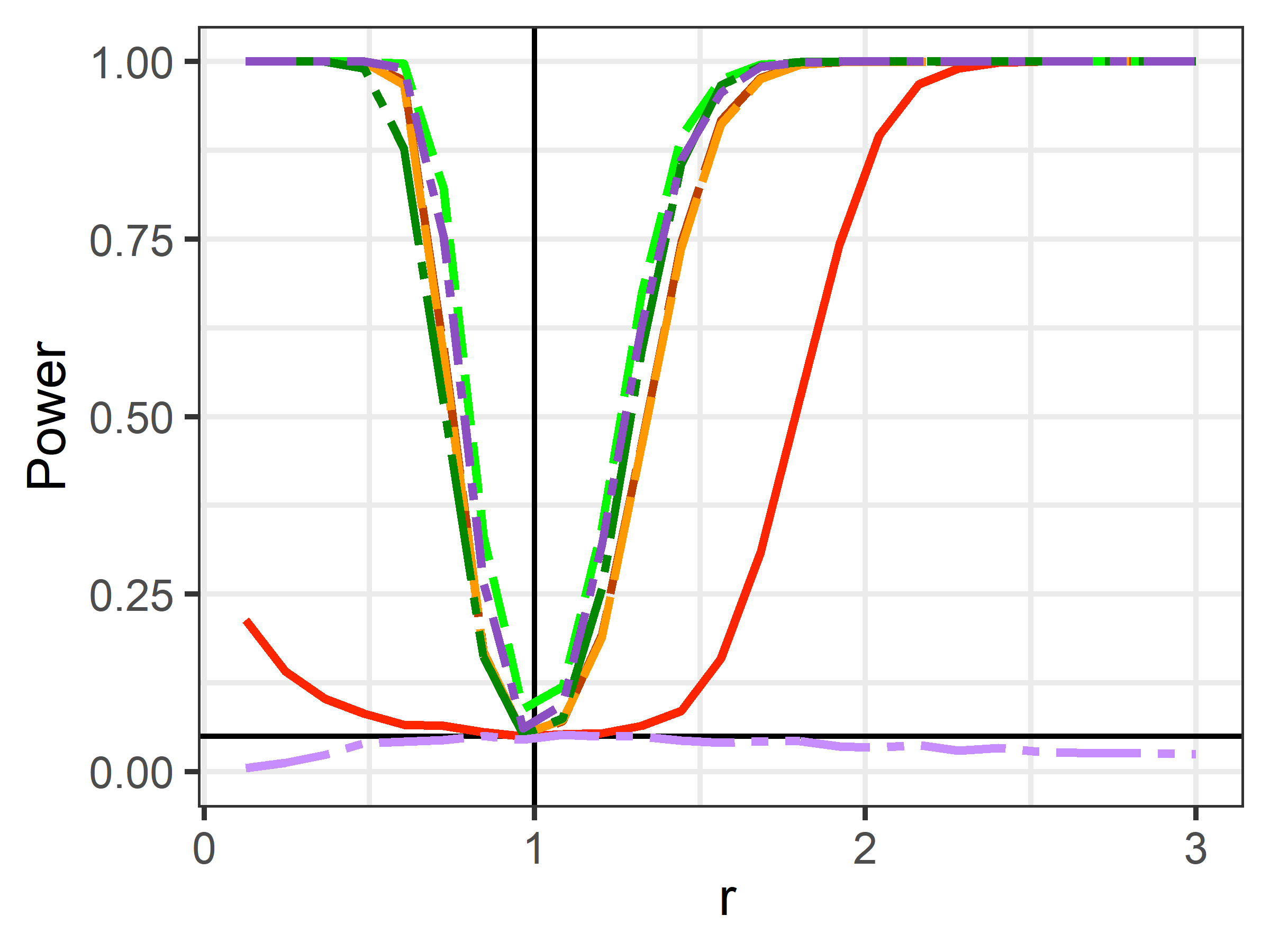}
        \caption{Study 2: Variances}
    \end{subfigure}
    \begin{subfigure}[t]{0.4\textwidth}
        \centering
        \includegraphics[width=\textwidth]{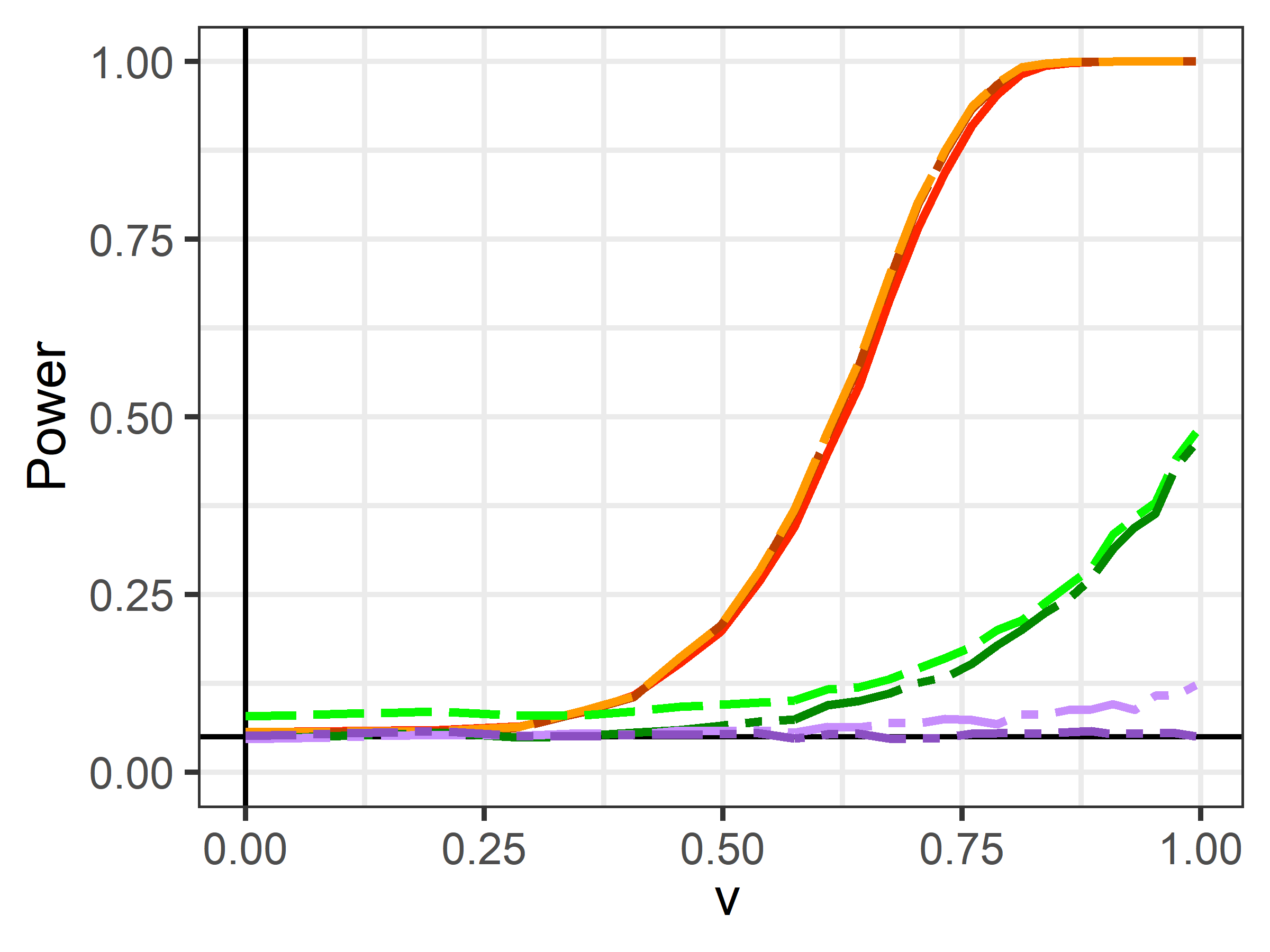}
        \caption{Study 3: Covariances}
    \end{subfigure}%
    \begin{subfigure}[t]{0.4\textwidth}
        \centering
        \includegraphics[width=\textwidth]{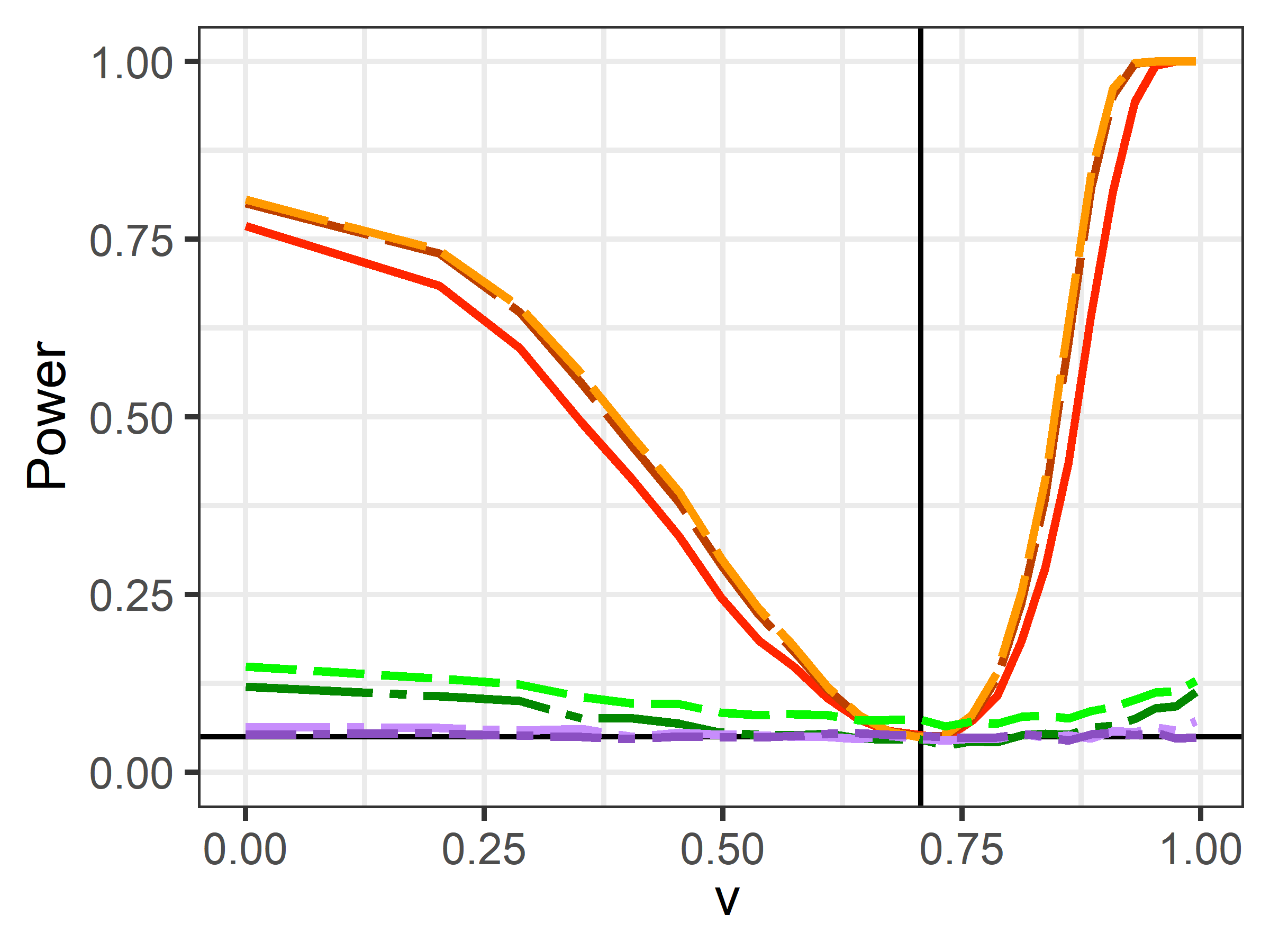}
        \caption{Study 4: Covariances}
    \end{subfigure}
    \begin{subfigure}[t]{0.4\textwidth}
        \centering
        \includegraphics[width=\textwidth]{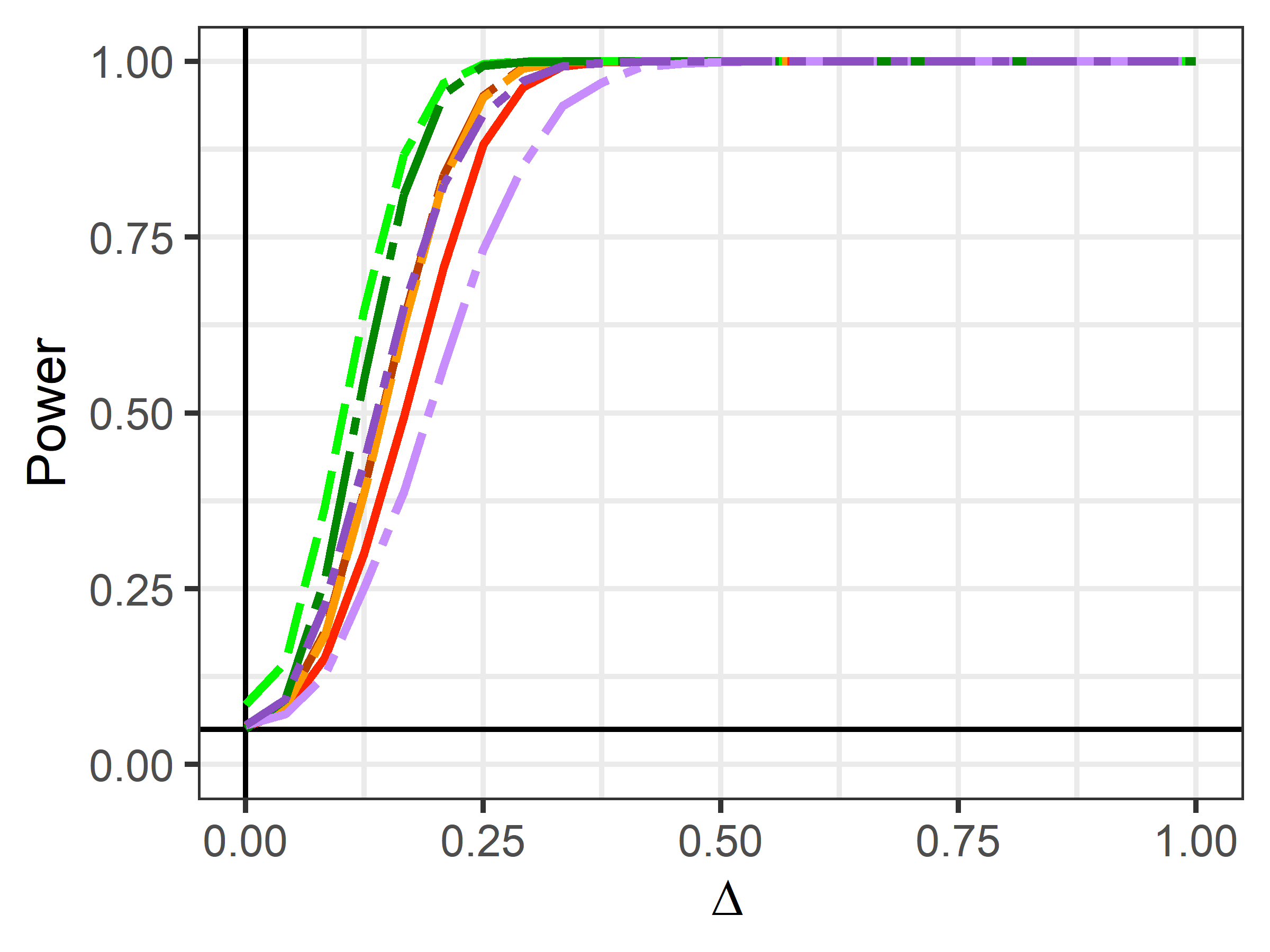}
        \caption{Study 5: Composite}
    \end{subfigure}%
    \begin{subfigure}[t]{0.4\textwidth}
        \centering
        \includegraphics[width=0.50\textwidth]{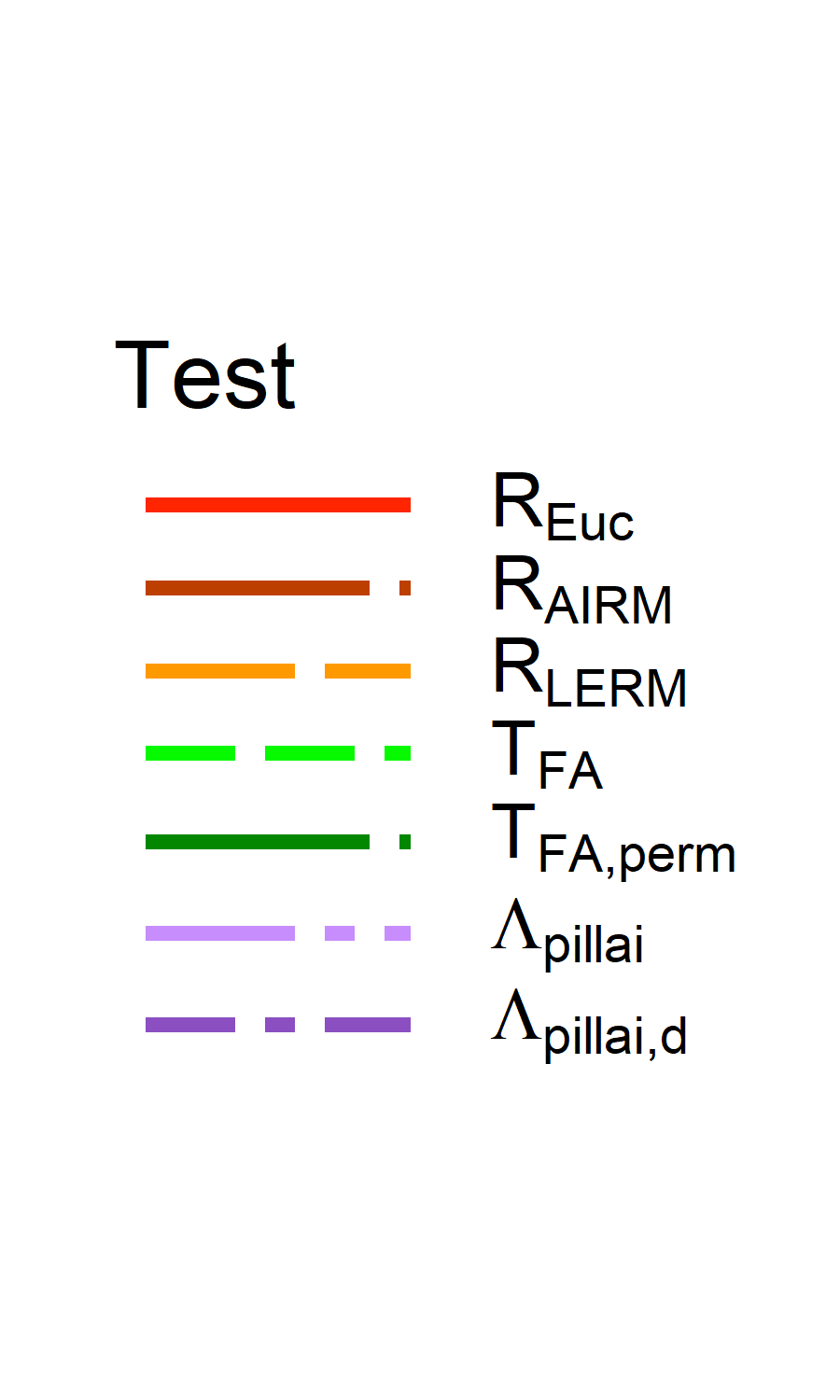}
    \end{subfigure}%
    \caption{Results from each simulation study in scenario 1. In each plot, a vertical black line indicates the null hypothesis, and a horizontal black line indicates the desired Type I error rate, $\alpha=0.05$.}
    \label{fig:sms_scen1}
\end{figure}

The simulation parameters for scenario 1 are summarized in Table \ref{tab:scen1_studies} and the results of the simulation studies are shown in Figure \ref{fig:sms_scen1}.
When the group means differ in study 1, all tests except $\bm{\Lambda_{pillai,d}}$ show similar ability to reject $H_0$ at varying effect sizes. 
Since $\bm{\Lambda_{pillai,d}}$ only works with distances, it cannot pick up on a shift in the mean of $X_2$, so it rejects with rate $\alpha=0.05$ at all effect sizes.

When changing group variances in study 2, however, $\bm{\Lambda_{pillai,d}}$ is able to detect the difference, while $\bm{\Lambda_{pillai}}$ is not.
Recall the latter test, being reconfigured from the Euclidean Pillai-Bartlett MANOVA test, is only designed to detect differences in means.
Both the Riemannian ($\bm{R_{Euc}}$, $\bm{R_{AIRM}}$, $\bm{R_{LERM}}$) and Fr\'echet ANOVA based tests ($\bm{T_{FA}}$ and $\bm{T_{FA,perm}}$) reject under departure of the null as well, though the Riemannian test using the underlying Euclidean Metric requires a much more drastic group difference to detect the difference.
This shows the utility of the LERM and AIRM Riemannian geometries when comparing PSD matrices in certain scenarios.

For studies 3 and 4, the Riemannian tests show the greatest sensitivity to changes in dependence.
Through additional simulations (not shown), we identified that the statistics which compare correlation matrices ($R_{P,Euc}$, $R_{P,AIRM}$, $R_{P,LERM}$) are the the most sensitive to the differences in dependence in these studies.
Omitting the correlation statistics from the Riemannian metric based tests drastically decreases their power against this alternative.
The Fr\'echet ANOVA based tests ($\bm{T_{FA}}$, $\bm{T_{FA,perm}}$) include a statistic testing for changes in covariance, $T_{12}$, resulting in power against the alternatives in these studies, although substantially lower than that of the Riemannian tests.
The Pillai-Bartlett tests ($\bm{\Lambda_{pillai,d}}$, $\bm{\Lambda_{pillai}}$) show that adapting classical MANOVA procedures may not prove useful in some instances, with $\bm{\Lambda_{pillai}}$ showing low power even at the largest effect size of $v=1$ in scenario 3, and $\bm{\Lambda_{pillai,d}}$ rejecting close to $\alpha=0.05$.

Though studies 1-4 show that the tests considered can behave very differently under different types of departures from the null hypothesis, they perform similarly when groups 1 and 2 differ in several ways simultaneously as in study 5.
Presumably each test is leaning on its own strengths in this case, and the similar power curves are a function of the relative magnitude of change in mean, variance, and covariance affected by the chosen ranges of $\delta$, $r$, and $v$ respectively.

Across all studies, the non-Euclidean Riemannian metric based tests ($\bm{R_{LERM}}$, $\bm{R_{AIRM}}$) perform the most consistently well, being on par or superior to the other tests examined.

\subsection{Scenario 2: Random Networks with Node Covariates, Two Groups}

In the second scenario, we consider random networks with undirected, simple edges (no loops or multi-edges) with ten nodes each and a real valued covariate on each node.
Our interest is in detecting a difference between the groups with respect to network topology and/or node covariates, or some change in the relationship between topology and covariates.
Therefore we identify two metric spaces, one which captures network topology and one which captures the node covariates.
Let $\Omega_1$ consist of graph Laplacians equipped with the Frobenius metric, and $\Omega_2$ consist of vectors in $\mathbb{R}^{10}$ equipped with the standard Euclidean metric, so that each network can be represented by a random object vector $(X_1, X_2) \in \Omega_1 \times \Omega_2$.

\begin{table}[h]
    \centering
    \begin{tabular}{l|c|c}
        \textbf{Study / Group} & $\bm{\gamma}$ & $\bm{\nu}$ \\
        \hline
        Study 1: Power Law & & \\
        \quad - Group 1 & $\bm{\in (2, 3)}$ & $1$\\
        \quad - Group 2 & $\bm{2.5}$ & $1$ \\
        Study 2: Power Law & & \\
        \quad - Group 1 & $\bm{\in (-1, 2)}$ & $1$\\
        \quad - Group 2 & $\bm{1}$ & $1$ \\
        Study 3: Dependence & & \\
        \quad - Group 1 & $2.5$ & $\bm{1}$ \\
        \quad - Group 2 & $2.5$ & $\bm{\in (0.125, 3)}$\\
        Study 4: Composite & & \\
        \quad - Group 1 & $\bm{\in (2.5, 3)}$ & $\bm{1}$ \\
        \quad - Group 2 & $\bm{2.5}$ & $\bm{\in (1, 3)}$ \\
    \end{tabular}
    \caption{Summary of simulation parameters for Groups 1 and 2 in each study of scenario 2. Bold sections highlight the differences between groups in each study.}
    \label{tab:scen2_studies}
\end{table}

We generate observations of $(X_1, X_2)$ by first sampling $X_1$, the network topology, then sampling $X_2$, the node covariates, conditional on $X_1$.
We sample random networks using the Barabasi-Albert growth-plus-preferential-attachment mechanism~\cite{barabasi_emergence_1999} in which each network starts with a single node, then nodes are added one at a time, each connecting ("attaching") to a single existing node at random.
The probability of attaching to an existing node is proportional to $k^\gamma$ with $\gamma \in \mathbb{R}$, where $k$ is the current number of edges currently incident to that node, its \emph{degree}.
Note that a node's degree $k$ can change as the network is constructed, so we let $k_f$ denote a node's final degree after the network is fully sampled.
Values of $\gamma$ closer to $0$ approach more uniform degree distributions while values further from $0$ tend to produce more skewed degree distributions.
Each node covariate is generated according to a Gamma distribution with mean $k_f$, and variance $\nu$.
Because network topology and node covariates both depend on node degree $k_f$, there is an inherent dependence between $X_1$ and $X_2$, and the strength of that dependence depends on the magnitudes of $\gamma$ and $\nu$: larger magnitudes of $\gamma$ and smaller values of $\nu$ correspond with stronger dependence between $X_1$ and $X_2$.
We consider four studies, summarized in Table \ref{tab:scen2_studies}, and for each, both groups have sample size $n = 100$.

For the first study, $\gamma$ is fixed at  $2.5$ for group 2 while it varies between $2$ and $3$ for group 1.
Both groups have $\nu = 1$.
The second study is identical to study 1, except that $\gamma$ is fixed at $1$ for group 2 while it varies between $-1$ and $2$ for group 1.
Changing $\gamma$ will affect the mean and variance of $X_1$ (network topology) as well as the mean of $X_2$ (node covariates), leading to a multifaceted departure from the null hypothesis.

In the third study, $\gamma = 2.5$ for both groups, but $\nu$ is fixed at $1$ for group 1 and varies between $0.125$ and $3$ for group 2.
Higher values of $\nu$ lead to weaker dependence between $X_1$ and $X_2$.
In the fourth study both groups vary at the same time, but in different ways.
For group 1, $\gamma$ varies between $2.5$ and $3$ with $\nu$ fixed at $1$, while for group 2, $\nu$ varies between $1$ and $3$ simultaneously with $\gamma$ fixed at $2.5$.
Specifically, $(\gamma_\text{group 1}, \nu_\text{group 2}) = (1-\Delta)\times (2.5, 1) + \Delta (3, 3)$, with $\Delta \in [0,1]$ indicating the overall ``effect size''.
\begin{figure}
    \centering
    \begin{subfigure}[t]{0.4\textwidth}
        \centering
        \includegraphics[width=\textwidth]{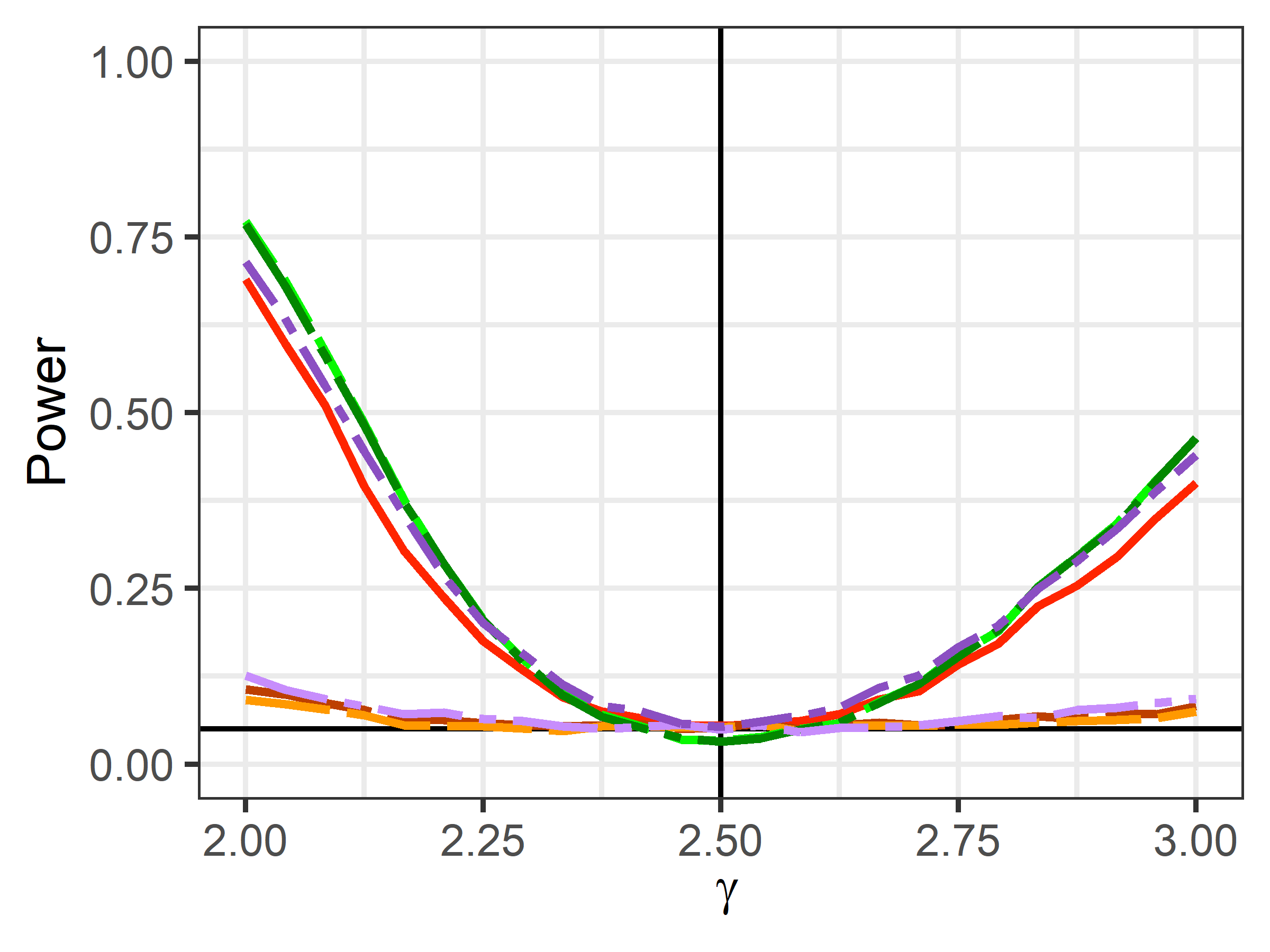}
        \caption{Study 1: Topology 1}
    \end{subfigure}%
    \begin{subfigure}[t]{0.4\textwidth}
        \centering
        \includegraphics[width=\textwidth]{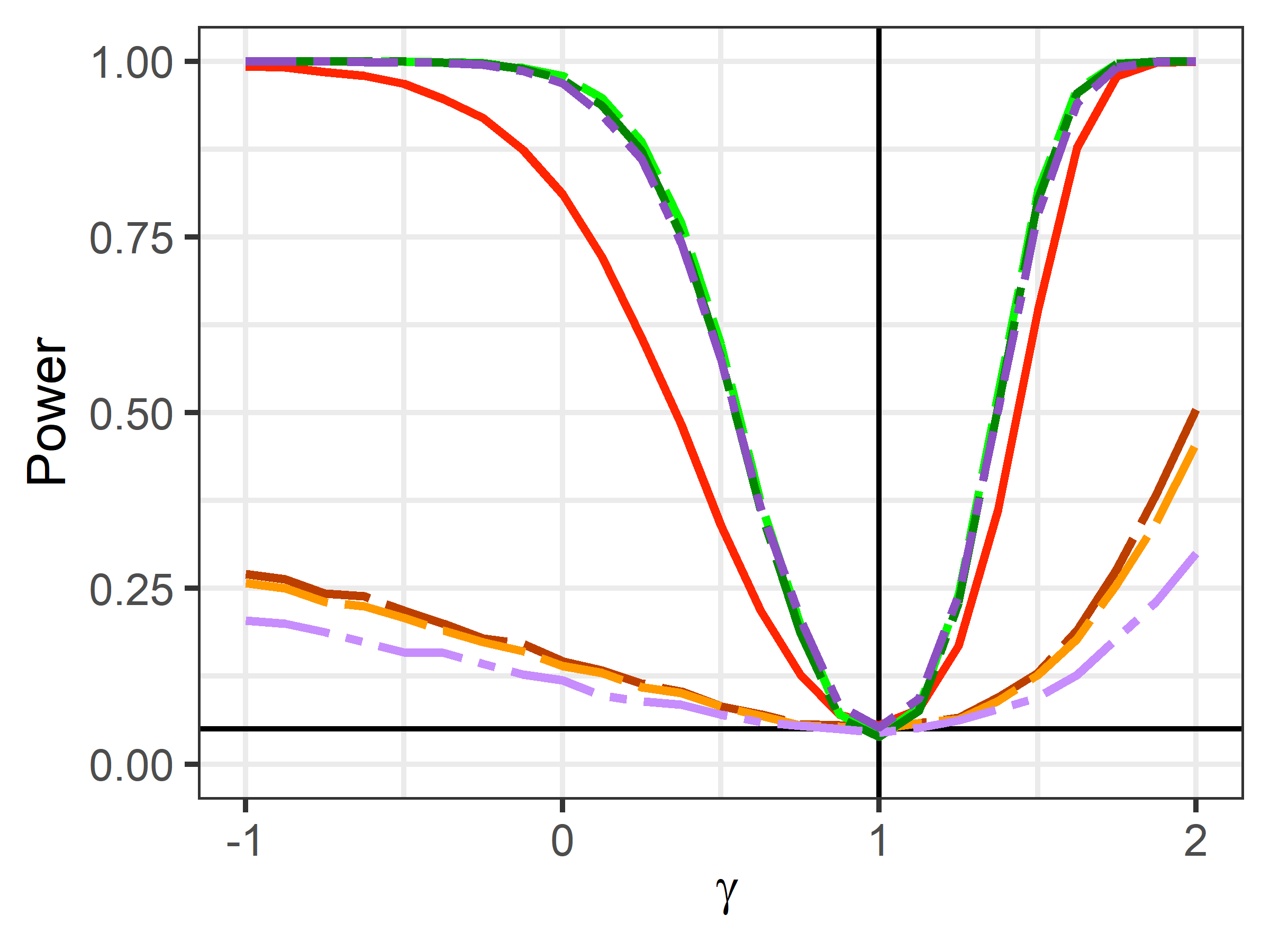}
        \caption{Study 2: Topology 2}
    \end{subfigure}
    \begin{subfigure}[t]{0.4\textwidth}
        \centering
        \includegraphics[width=\textwidth]{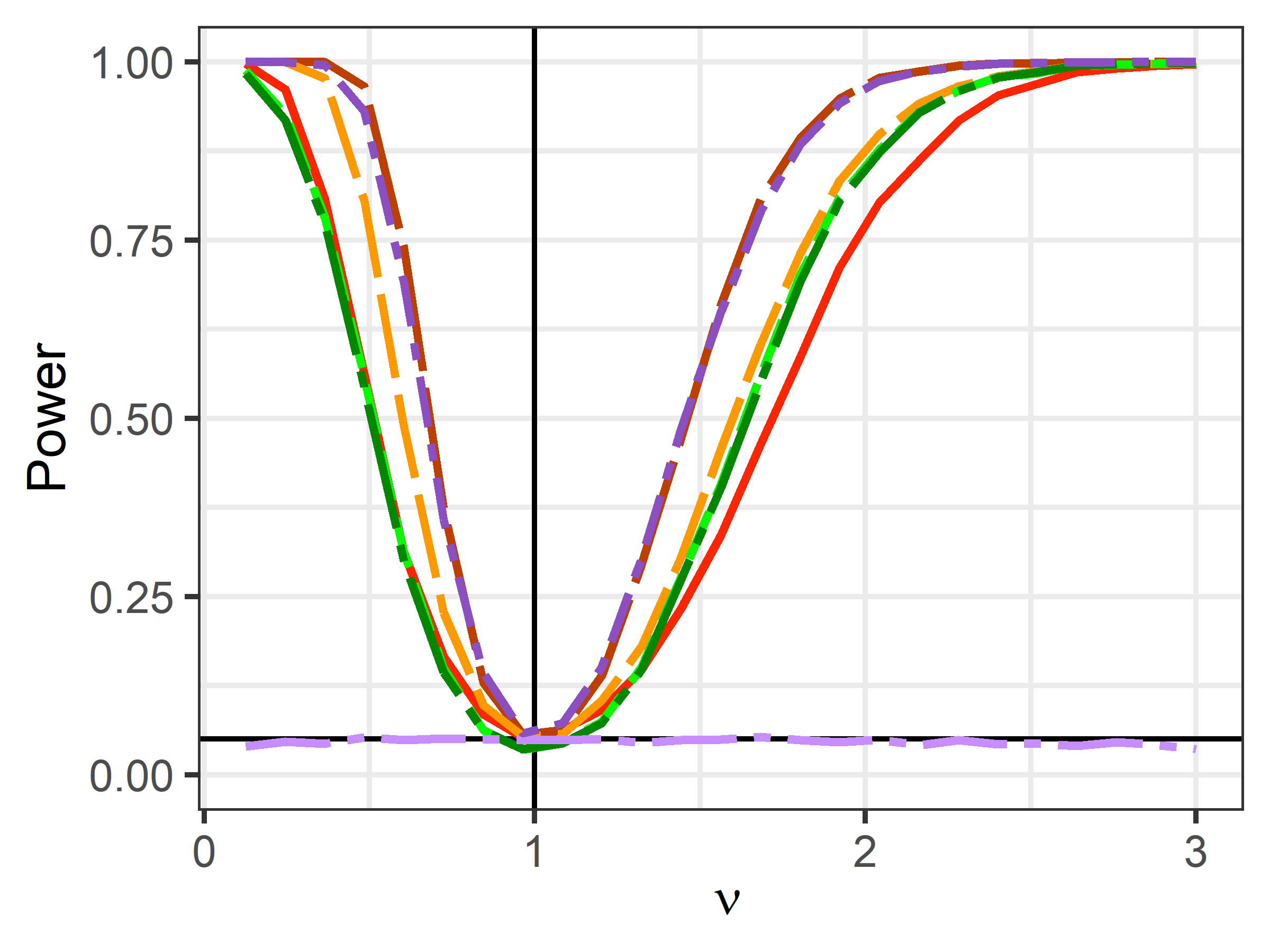}
        \caption{Study 3: Node Features}
    \end{subfigure}%
    \begin{subfigure}[t]{0.4\textwidth}
        \centering
        \includegraphics[width=\textwidth]{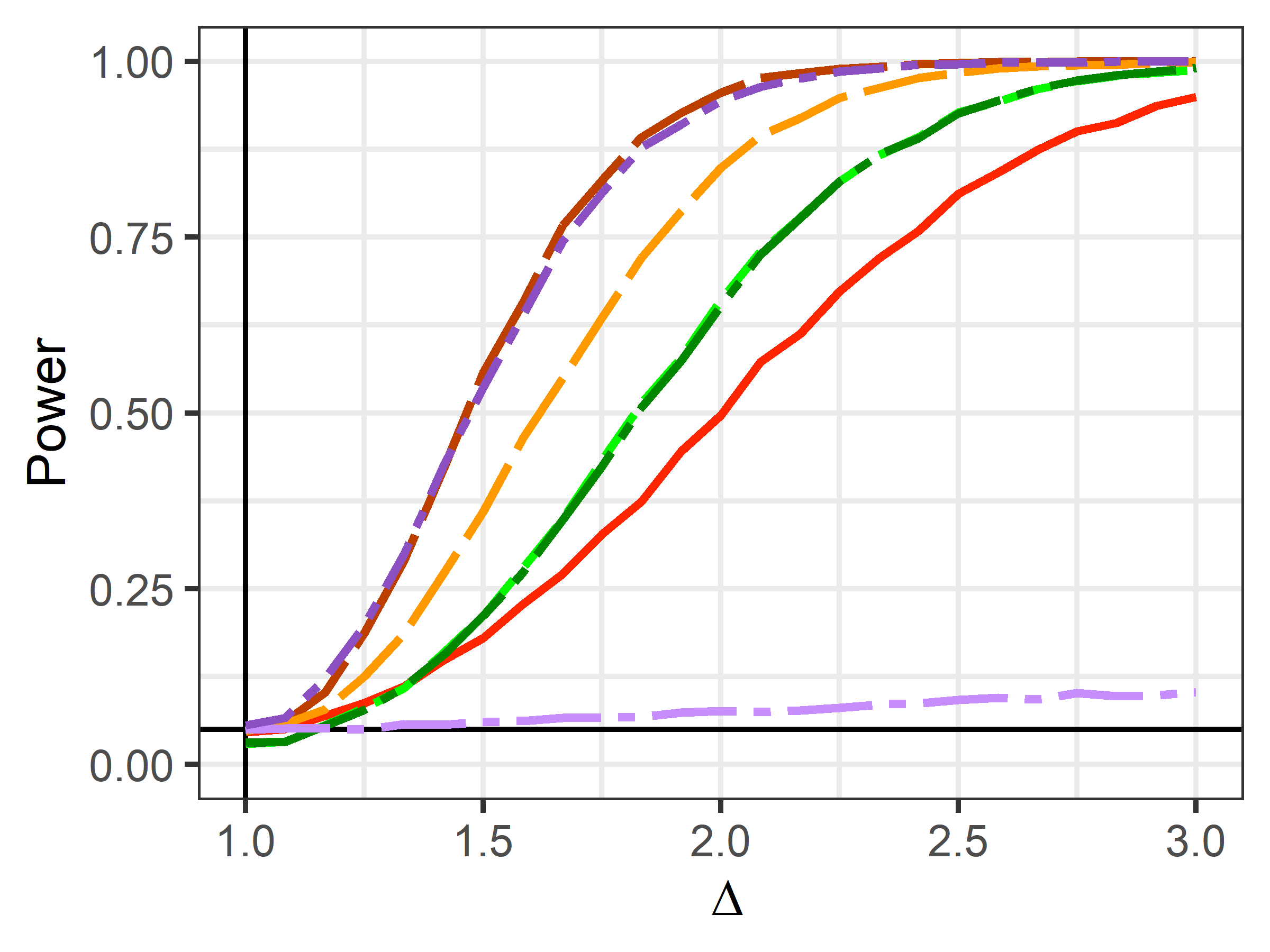}
        \caption{Study 4: Composite}
    \end{subfigure}
    \begin{subfigure}[t]{0.6\textwidth}
        \centering
        \includegraphics[width=0.99\textwidth]{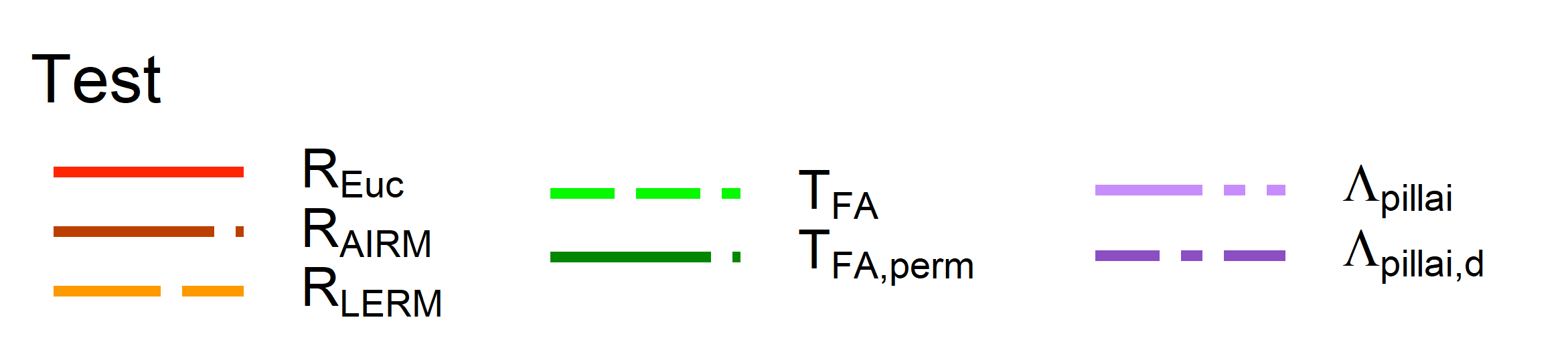}
    \end{subfigure}
    \caption{Results from each simulation study in scenario 2. In each plot, a vertical black line indicates the null hypothesis, and a horizontal black line indicates the desired Type I error rate, $\alpha=0.05$}
    \label{fig:sms_scen2}
\end{figure}

The results of the simulation studies are shown in Figure \ref{fig:sms_scen2}.
In study 1 the changing topology is best detected by the Fr\'echet ANOVA based tests, the Riemannian test using the underlying Euclidean metric, and the Pillai-Bartlett trace using distances.
In this case, the non-Euclidean geometry imposed by the AIRM and LERM metrics works against those tests.
It turns out that those tests are most sensitive to changes in determinants, owing to the behavior of the underlying $d_{LERM}$ and $d_{AIRM}$ distances. 
For instance, the \emph{shortest} path between two PSD matrices $A$ and $B$ with the same determinant will always follow a curve of constant determinant.
Empirically, changing $\gamma$ results in group covariance matrices which are different but have similar determinant, similarly for the pooled and mean group covariance matrices, meaning this study is particularly challenging for the $\bm{R_{LERM}}$ and $\bm{R_{AIRM}}$ tests.

Study 2 shows that $\bm{R_{LERM}}$ and $\bm{R_{AIRM}}$ have \emph{some} ability to detect changes in $\gamma$, but are much less powerful that other tests.
Study 3 reflects a more isolated difference between groups, where only the variance of $X_2$ and covariance between $X_1$ and $X_2$ change.
In this case, the $\bm{R_{AIRM}}$ and $\bm{R_{LERM}}$ are comparable to or are better than the Fr\'echet ANOVA based tests.

When combining both a changing $\gamma$ and a changing $\nu$ in study 4, we see that the $\bm{R_{AIRM}}$ and $\bm{\Lambda_{pillai,d}}$ tests have the greatest sensitivity to change, but again this is influenced by the relative magnitude of change of $\gamma$ and $\nu$.

In all four networks studies, the distance based Pillai-Bartlett test is among the best performing of all tests considered, which makes sense given that all studies involve changing variances and hence changing distances between the groups.
On the other hand, the adapted Euclidean Pillai-Bartlett test shows some power against changing $\gamma$, which affects the mean of $X_1$, but is not sensitive to changing $\nu$ at all, which affects only variances and covariances.

\subsection{Type I Error Rate}

\begin{table}[]
    \centering
    \begin{tabular}{c|c|c|c|c|c|c|c}
        \bf{Scenario} & $\bm{R_{Euc}}$ & $\bm{R_{AIRM}}$ & $\bm{R_{LERM}}$  & $\bm{T_{FA}}$ & $\bm{T_{FA,perm}}$ & $\bm{\Lambda_{pillai}}$& $\bm{\Lambda_{pillai,d}}$ \\
        \hline
         1 & 0.046 & 0.048 & 0.047 & 0.074 & 0.041 & 0.049 & 0.050\\
         2 & 0.045 & 0.046 & 0.044 & 0.033 & 0.030 & 0.051 & 0.051\\
    \end{tabular}
    \caption{Type I error rates for each test in the studies of scenarios 1 and 2. The Type I error is not a function of the alternative hypothesis so it is the same across all studies within each scenario. We simulated 6,000 data sets under each scenario, giving a standard error of approximately 0.003 for each estimate, and used 1,000 permutations for each permutation test. .}
    \label{tab:type1_err}
\end{table}

Table \ref{tab:type1_err} summarizes the Type I error for scenarios 1 and 2.
All tests are below the desired Type I error rate of 0.05 in both scenarios, with the exception of the Fr\'echet ANOVA based $\bm{T_{FA}}$ in Scenario 1, at 0.074.
This statistic relies on asymptotic results for the Fr\'echet variance, and presumptive (not yet proven) asymptotic results for the Fr\'echet covariance, so the over-rejection is likely explained by a lack of convergence to the asymptotic distribution in the tail of the Fr\'echet ANOVA statistics.
Indeed, the same over-rejection is not seen in the permutation version of the test.

The conservative nature of the Bonferronni correction can be seen in both the Riemannian based tests, $\bm{R_{Euc}}$, $\bm{R_{AIRM}}$, $\bm{R_{LERM}}$, and the Fr\'echet ANOVA based tests, $\bm{T_{FA}}$ and $\bm{T_{FA,perm}}$, which reject below the nominal level.
As the number of metric spaces $S$ increases, the Fr\'echet ANOVA based tests must correct for a greater number of underlying tests, while the number of underlying test statistics for the Riemannian based tests remain at three.
The lowest Type I error rate is seen for $\bm{T_{FA}}$ and $\bm{T_{FA,perm}}$ in scenario 2, resulting from correlation between the underlying test statistics.
We expect correlation to be common among the Fr\'echet ANOVA based statistics due to the correlation between the estimated elements of the Fr\'echet covariance matrices.
The Pillai tests, $\bm{\Lambda_{pillai}}$ and $\bm{\Lambda_{pillai,d}}$, don't use a Bonforroni correction, and appear to reject at the nominal level.

\subsection{Effect of Sample Size and Group Balance}

\begin{figure}
    \centering
    \includegraphics[width=\textwidth]{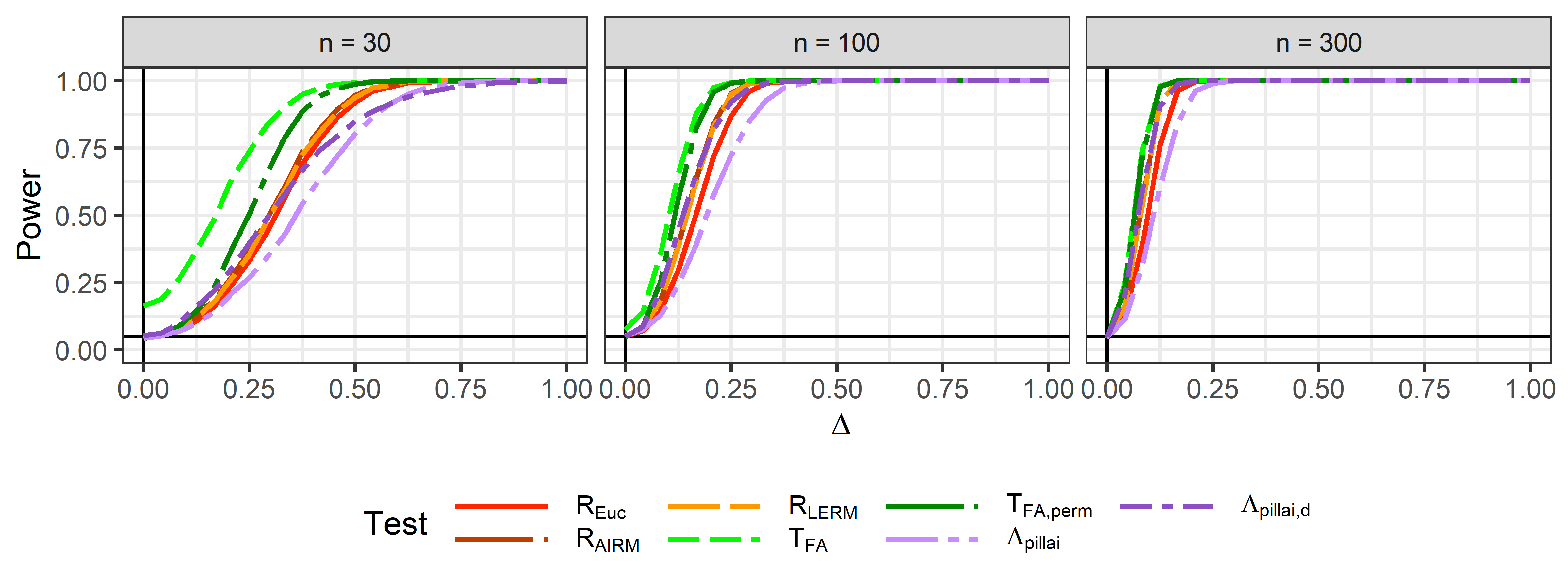}
    \caption{Results of study 5 in scenario 1 at three different sample sizes (equal for each group). In each plot, a vertical black line indicates the null hypothesis, and a horizontal black line indicates the desired Type I error rate, $\alpha=0.05$.}
    \label{fig:sms_ss}
\end{figure}

\begin{figure}
    \centering
    \includegraphics[width=\textwidth]{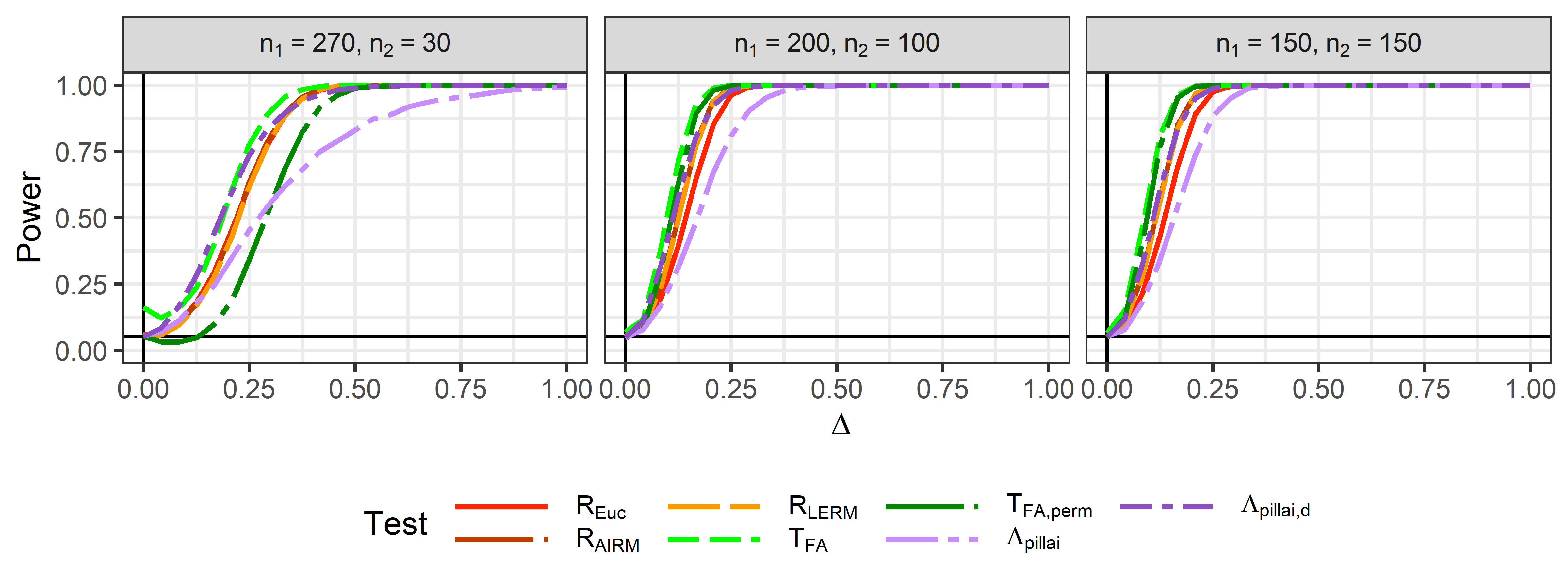}
    \caption{Results of study 5 in scenario 1 with varying balance (unequal samples sizes). In each plot, a vertical black line indicates the null hypothesis, and a horizontal black line indicates the desired Type I error rate, $\alpha=0.05$.}
    \label{fig:sms_im}
\end{figure}

Finally, we investigate how sample size and balanced group sizes affect the performance of each test in study 5 of scenario 1.
Figure \ref{fig:sms_ss} shows power curves under three different sample sizes: $n_1 = n_2 = 30$, $n_1 = n_2 = 100$, and $n_1 = n_2 = 300$.
As expected, larger sample sizes lead to increased power at all effect sizes across all tests examined.
However we see the inflated Type I error rate (shown at $\Delta=0$ in Figure \ref{fig:sms_ss}) for the asymptotic Fr\'echet ANOVA based test, $\bm{T_{FA}}$, unsurprising given its reliance on large sample asymptotics.
However, the permutation version of the Fr\'echet ANOVA based test, $\bm{T_{FA,perm}}$, shows a more appropriate Type I error and the best power among the tests (excluding $\bm{T_{FA}}$) at the lower sample sizes.

Figure \ref{fig:sms_im} shows the results of varying the relative sample size of each group, keeping the total sample size the same ($n_1 + n_2 = 300)$.
We consider highly imbalanced ($n_1 = 270,~n_2 = 30$), moderately imbalanced ($n_1=100, n_2=200$), and balanced ($n_1 = 150, n_2 = 150$) situations.
The inflated Type I error rate ($\Delta = 0$) of $\bm{T_{FA}}$ is again present in the highly imbalanced case, for the same reason as in the balanced design with $n_1 = n_2 =30$, since group 2 has few observations.
Among the other tests, the Riemannian tests and distance based Pillai-Bartlett tests are best in the highly imbalanced case.

\section{Conclusions}

In expanding the setting of random objects to multiple metric spaces, we defined the Fr\'echet analog to covariance.
We motivated two possible definitions and explored their interpretations, as well as established consistency for the sample estimator for one of them, which allowed the construction of a full metric space covariance matrix.
Using this matrix, we considered various tests for differences in mean and covariance structure in multiple metric spaces between two or more groups.
Finally, we examined the performance of these tests under two simulation scenarios and studied their power characteristics under a variety of  departures from the null distribution.

It is clear from our simulations that no single proposed test is universally preferred, a result reminiscent of prior work comparing the performance of the various classical MANOVA tests~\cite{finch_comparison_2005, ates_comparison_2019}.
However, the adaptations of the Pillai-Bartlett test appeared to fall short of the performance of the other tests: $\bm{\Lambda_{pillai,d}}$ because it only sees the within group distances, $\bm{\Lambda_{pillai}}$ because it is not designed to test for variances. 
A possible augmentation to either $\bm{\Lambda_{pillai,d}}$ or $\bm{\Lambda_{pillai}}$ would be to incorporate some adaptation of Box's M test~\cite{box_ra_1980}, though performance may be poor when normality assumptions do not hold, as noted in~\cite{manly_multivariate_2016}.

Permutation tests provided a path to testing for group differences in the absence of asymptotic results, which is particularly useful when distributional properties are unknown.
However, the desire remains to establish asymptotic results for the Fr\'echet covariance, and theoretically justify the Riemannian approaches, if possible showing asymptotically favorable power as sample size grows.

Despite recent progress on working with random objects in metric spaces, many existing challenges remain. 
Among these are the selection of an appropriate metric, the mechanic of its computation in various metric spaces of interest, and efficient estimation of the sample Fr\'echet mean.
Regardless, future prospects of the metric space framing of data sets appear promising.

\bibliographystyle{unsrt}  
\bibliography{references}  

\appendix
\section{Metric Spaces}

A \emph{metric} on a set $\Omega$ (not to be confused with $\Omega(p,q)$, the space of all binary matrices satisfying margins $p$ and $q$) is a function $d:\Omega \times \Omega \rightarrow [0, \infty)$ such that, for all $\omega_1, \omega_2, \omega_3 \in \Omega$, $d$ satisfies the following:
\begin{enumerate}
    \item $d(\omega_1, \omega_2) = 0 \Leftrightarrow \omega_1 = \omega_2$
    \item $d(\omega_1, \omega_2) = d(\omega_2, \omega_1)$
    \item $d(\omega_1, \omega_2) \leq d(\omega_1, \omega_3) + d(\omega_3, \omega_2)$.
\end{enumerate}
A \emph{metric space} is a non-empty set endowed with a metric.
When there are multiple metric spaces, we given them subscripts, for example $\Omega_1, \ldots, \Omega_D$.

Some well known metrics that we use in this dissertation include the Wasserstein metric between probability spaces, and the Frobenius metric between matrices.
Consider the space $M$ of probability distributions for which every Borel probability measure on $M$ is also a Radon measure, and assume a metric $d$ is defined for this space.
Then let $P_p(M)$ consist of all probability measures on $M$ with finite $p$th moment. 
The \emph{Wasserstein-p metric} between two probability measures $\mu$ and $\nu$ in $P_p(M)$ is defined as
\begin{align}
    W_p(\mu, \nu) = \text{inf}~\left(E[d(X,Y)^p] ^{1/p}\right),
\end{align}
where the infimum is taken over all joint distributions of $X$ and $Y$ with marginals $\mu$ and $\nu$, respectively.
For example, comparing univariate normal distributions with identical variance, say $X \sim N(a, 1)$ and $Y \sim N(b, 1)$, the Wasserstein-2 metric is simply the absolute difference in their means, $|a-b|$.

The \emph{Frobenius norm} of a real valued $m \times n$ matrix $A$ is 
\begin{align}
    ||A||_F = \sqrt{\sum_{i=1}^m \sum_{j=1}^n a_{ij}^2} = \sqrt{\sum_{i}^{\text{min}(m, n)} \lambda_i},
\end{align}
where $\lambda_i$ are the eigenvalues of $A^T A$.
This norm admits a metric between two $m \times n$ matrices $A$ and $B$ :$d(A, B) = ||A-B||_F$.

\section{Riemannian Manifolds}
\label{chap_sec:riem_manifolds}

A \emph{topological space} is an ordered pair $(X, \tau)$ where $X$ is a set of points and $\tau$ is a collection of subsets of $X$ satisfying:
\begin{enumerate}
    \item $X \in \tau$ and $\emptyset \in \tau$,
    \item $\tau$ is closed under countable unions: that is, for $A_1, \ldots \in \tau$, $\bigcup_i A_i \in \tau$,
    \item $\tau$ is closed under finite intersections, that is, for $A_1, \ldots, A_n \in \tau$, $\bigcap_i A_i \in \tau$.
\end{enumerate}
The collection $\tau$ is called a topology on $X$, and the elements of $\tau$ are called open subsets.
A \emph{homeomorphism} between two topological spaces $T_1$ and $T_2$ is a function $f: T_1 \rightarrow T_2$ such that $f$ is a bijection, $f$ is continuous, and $f^{-1}$ is continuous.
Loosely speaking, the existence of a homeomorphism implies that $T_1$ can be continuously deformed into $T_2$, and vice versa.

A \emph{manifold} $M$ is a topological space where for each point in the topology, there exists a local neighborhood which is homeomorphic to an open subset of $\mathbb{R}^n$ for some $n$.
Manifolds may be described using a set of \emph{charts}, $A = \{(U_\alpha, \phi_\alpha)\}_\alpha$, each of which consists of an open subset $U_\alpha \in M$ and a homeomorphism $\phi_\alpha$ from $U_\alpha$ to a subset of $\mathbb{R}^n$.
The subsets $U_\alpha$ cover $M$, that is, $\bigcup_\alpha U_\alpha = M$, and $A$ is called an \emph{atlas}.
Where any two open subsets $U_\alpha$ and $U_\beta$ overlap, we can define the \emph{transition map} $\phi_\beta^\alpha = \phi_\beta \phi_\alpha^{-1}$.
An atlas where all transition maps are differentiable is called a \emph{differentiable atlas}.
This means that overlapping charts have ``similar'' maps to $\mathbb{R}^n$ in their overlap.
A chart $(U, \phi)$ is \emph{compatible} with a differentiable atlas if its inclusion in the atlas results in a differentiable atlas, and a maximal differentiable atlas consists of all charts which are compatable with the given atlas.

A \emph{differentiable manifold} $M$ is a topological space with three additional properties: it is Hausdorff, second countable, and possesses a maximal differentiable atlas.
To be Hausdorff, any two distinct points in $M$ must have disjoint neighborhoods.
To be second countable, the topology of $M$ must have a countable base; that is, there exists some countable collection $\{U_i\}_{i=1}^\infty$ with $U_i \in \tau$ such that any element in $\tau$ is a union of some set of $U_i$'s. 
Differentiable manifolds allow for the construction of a calculus on the topological space, since each chart is homeomorphic to $\mathbb{R}^n$ and all transition maps are differentiable as well.
This leads to the idea of tangent vectors at a point on the manifold, and a \emph{tangent space} at each point on the manifold, which is a vector space consisting of all tangent vectors at that point.

A \emph{Riemannian manifold} $M$ is a differentiable manifold in which the tangent space at each point $p \in M$ is equipped with a positive definite inner product, and the inner product varies smoothly with the point $p$.
Any vector space with an inner product admits a norm and a corresponding metric, so every tangent space in a Riemannian manifold is also equipped with a \emph{Riemannian metric} that smoothly varies with $p \in M$.
If $M$ is connected (i.e. it cannot be expressed as a disjoint union of two nonempty open subsets), then it is possible to define the distance along a path between points $p, q \in M$ by integrating the Riemannian metric along that path.
The shortest possible path length is the \emph{geodesic}, and defines a metric on $M$, so every connected Riemannian manifold is also a metric space.
Different choices of Riemannian metric (at a point $p$) will result in different metrics on the space $M$.

The space of symmetric positive definite (SPD) matrices can be thought of as a cone in Euclidean space, hence it is also a differentiable manifold.
There are several choices for Riemannian metric, and 
each metric gives rise to a different shortest path and shortest path length between two symmetric positive definite matrices.
Three common choices are the standard Euclidean metric (Euc), the affine invariant Riemannian metric (AIRM), and the log-Euclidean Riemannian metric (LERM)~\cite{pennec_riemannian_2006, arsigny_geometric_2007, you_re-visiting_2021, moakher_differential_2005}.
The Euclidean Riemannian metric implies the distance between two SPD matrices $A$ and $B$ is 
\begin{equation}
    d_{Euc}(A, B) = || A - B ||_F = \sqrt{\sum_i \lambda_i^2\left(A-B\right)}\label{eq:EUC_metric},
\end{equation}
where $||\cdot||_F$ is the Frobenius norm.
As noted in \cite{pennec_riemannian_2006, lin_riemannian_2019}, this metric can lead to determinant swelling when used to compute means, that is, the Fr\'echet mean of a set of SPD matrices may have determinant larger than any of the individual matrices.
They also note that shortest paths between two matrices of the same determinant will not necessarily preserve the determinant along the path.
This is used to motivate the AIRM and LERM Riemannian metrics, which do not suffer from the same defficiencies.

In the AIRM geometry, the distance between two SPD matrices $A$ and $B$ is given by
\begin{equation}
    d_{AIRM}(A, B) = || \text{Log}\left(A^{-1/2}B A^{-1/2}\right) ||_F = \sqrt{\sum_i \text{log}^2\left[\lambda_i(A^{-1/2} B A^{-1/2})\right]},
    \label{eq:AIRM_metric}
\end{equation}
where Log denotes the principal matrix logarithm (the inverse of the exponential map).
The \textit{affine invariant} naming comes from the fact that the underlying Riemannian metric is invariant to affine transformations (see \cite{pennec_riemannian_2006} for details).
Note that because $A^{-1/2} B A^{-1/2}$ and $A^{-1} B$ are similar matrices, they have the same eigenvalues, which allows us to avoid taking the square root of $A^{-1}$ when computing $d_{AIRM}$.

In the LERM geometry, the distance between two SPD matrices $A$ and $B$ is given by 
\begin{equation}
    d_{LERM}(A, B) = || \text{Log}(A) - \text{Log}(B) ||_F\label{eq:LERM_metric},
\end{equation}
which is simply the Euclidean metric after $A$ and $B$ are mapped to the set of symmetric matrices via the principal matrix logarithm, hence the name.

\section{Matrix Logarithm}
The LERM and AIRM Riemannian metrics lead to distances on the space of symmetric positive definite matrices which rely on the matrix logarithm.
The exponential map of an $n \times n$ matrix $A$ is given by the power series
\begin{align}
    \exp{A} = \sum_{k=0}^\infty \frac{A^k}{k!}.
\end{align}
Matrix B is a matrix logarithm of $A$ if $\exp{B} = A$.
Since positive definite matrix $A$ admits an eigenvalue decomposition, $A = U D U^T$, we can define the \emph{principal matrix logarithm} as 
\begin{align}
    \text{Log}(A) = U \text{Log}(D) U^T,
\end{align}
Where $\text{Log}(D)$ is a diagonal matrix where each element is the (scalar) natural log of the corresponding eigenvalue of $A$.

\section{Proof of Consistency (Theorem \ref{thm:cov_consistency})}
\begin{proof}[Proof of Theorem \ref{thm:cov_consistency}]
    Note that $d(\hat{\mu}_s, \mu_s)\overset{p}\rightarrow~0$ by the continuous mapping theorem and consistency of $\hat\mu_s$, similarly for $X_{s'}$.
    Then 
    \begin{align}
        |\widehat{\sigma}_{ss'} - \sigma_{ss'}| &= \left\lvert \frac{1}{n} \sum_{i=1}^n d(\hat{\mu}_s, X_{si})d(\hat{\mu}_{s'}, X_{s'i}) - E[d(\mu_s, X_s)d(\mu_{s'}, X_{s'})]\right\rvert \\
        &= |A + B + C | \\
        &\leq |A| + |B| + |C|
    \end{align}
    where 
    \begin{align}
        A &= \frac{1}{n} \sum_{i=1}^n \left[ d(X_{si}, \hat\mu_{s}) - d(X_{si}, \mu_{s}) \right] \left[ d(X_{s'i}, \hat\mu_{s'}) - d(X_{s'i}, \mu_{s'}) \right] \\
        B &= \frac{1}{n} \sum_{i=1}^n d(X_{si}, \hat\mu_{s})d(X_{s'i}, \mu_{s'}) + d(X_{si}, \mu_{s})d(X_{s'i}, \hat\mu_{s'}) - 2d(X_{si}, \mu_{s})d(X_{s'i}, \mu_{s'}) \\
        C &= \frac{1}{n} \sum_{i=1}^n d(X_{si}, \mu_{s})d(X_{s'i}, \mu_{s'}) - E\left[d(X_{si}, \mu_{s})d(X_{s'i}, \mu_{s'})\right]
    \end{align}
    Each of $|A|$, $|B|$, and $|C|$ converges to zero in probability as $n \rightarrow \infty$, as follows:
    \begin{align}
        |A| & \leq \frac{1}{n} \sum_{i=1}^n \left\lvert d(X_{si}, \hat\mu_{s}) - d(X_{si}, \mu_{s}) \right\rvert \left\lvert d(X_{s'i}, \hat\mu_{s'}) - d(X_{s'i}, \mu_{s'}) \right\rvert \\
        &\leq d(\hat\mu_{s}, \mu_{s}) d(\hat\mu_{s'}, \mu_{s'}).
    \end{align}
    Since $d(\hat\mu_{s}, \mu_{s})$ and $d(\hat\mu_{s'}, \mu_{s'})$ converge to zero in probability, so does $|A|$.
    For $B$,
    \begin{align}
        |B| &= \left\lvert \frac{1}{n} \sum_{i=1}^n \left[ d(X_{si}, \hat\mu_{s}) - d(X_{si}, \mu_{s}) \right] d(X_{s'i}, \mu_{s'}) + d(X_{si}, \mu_{s}) \left[ d(X_{s'i}, \hat\mu_{s'}) - d(X_{s'i}, \mu_{s'}) \right] \right\rvert \\
        &\leq \frac{1}{n} \sum_{i=1}^n \left\lvert d(X_{si}, \hat\mu_{s}) - d(X_{si}, \mu_{s}) \right\rvert d(X_{s'i}, \mu_{s'}) + \notag\\
        &\quad\quad\quad\frac{1}{n} \sum_{1=1}^n d(X_{si}, \mu_{s}) \left\lvert d(X_{s'i}, \hat\mu_{s'}) - d(X_{s'i}, \mu_{s'}) \right\rvert \\
        &\leq \frac{1}{n} \sum_{i=1}^n d(\hat\mu_{s}, \mu_{s}) d(X_{s'i}, \mu_{s'}) + d(X_{si}, \mu_{s}) d(\hat\mu_{s'}, \mu_{s'}),
    \end{align}
    which converges to zero in probability from the consistency of $\hat\mu_{s}$ and $\hat\mu_{s'}$, as well as the boundedness of both metric spaces.
    Finally, $|C|$ converges to zero in probability by the weak law of large numbers and the boundedness of the metric spaces.
    Consistency follows from Slutsky's theorem.
\end{proof}

\end{document}